\documentclass[sn-mathphys,Numbered]{sn-jnl}

\usepackage{graphicx}%
\usepackage{multirow}%
\usepackage{amsmath,amssymb,amsfonts}%
\usepackage{amsthm}%
\usepackage{mathrsfs}%
\usepackage[title]{appendix}%
\usepackage{xcolor}%
\usepackage{textcomp}%
\usepackage{manyfoot}%
\usepackage{booktabs}%
\usepackage{algorithm}%
\usepackage{algorithmicx}%
\usepackage{algpseudocode}%
\usepackage{listings}%

\newtheorem{theorem}{Theorem}
\newtheorem{lemma}{Lemma}

\begin{document}

\title[The biharmonic wave scattering problem]{Convergence of the PML method for the biharmonic wave scattering problem in periodic structures}


\author[1]{\fnm{Gang} \sur{Bao}}\email{baog@zju.edu.cn}
\equalcont{These authors contributed equally to this work}

\author*[2]{\fnm{Peijun} \sur{Li}}\email{lipeijun@math.purdue.edu}

\author[3]{\fnm{Xiaokai} \sur{Yuan}}\email{yuanxk@jlu.edu.cn}
\equalcont{These authors contributed equally to this work}

\affil[1]{\orgdiv{School of Mathematical Sciences}, \orgname{Zhejiang University}, \orgaddress{\city{Hangzhou}, 
\postcode{310027}, \state{Zhejiang}, \country{China}}}

\affil*[2]{\orgdiv{Department of Mathematics}, \orgname{Purdue University}, \orgaddress{ 
\city{West Lafayette}, \postcode{47907}, \state{Indiana}, \country{USA}}}

\affil[3]{\orgdiv{School of Mathematics}, \orgname{Jilin University}, \orgaddress{
\city{Changchun}, \postcode{130012}, \state{Jilin}, \country{China}}}


\abstract{This paper investigates the scattering of biharmonic waves by a one-dimensional periodic array of cavities embedded in an infinite elastic thin plate. The transparent boundary conditions are introduced to formulate the problem from an unbounded domain to a bounded one. The well-posedness of the associated variational problem is demonstrated utilizing the Fredholm alternative theorem. The perfectly matched layer (PML) method is employed to reformulate the original scattering problem, transforming it from an unbounded domain to a bounded one. The transparent boundary conditions for the PML problem are deduced, and the well-posedness of its variational problem is established. Moreover, exponential convergence is achieved between the solution of the PML problem and that of the original scattering problem.}




\maketitle

\section{Introduction}

Scattering of flexural waves in an elastic thin plate, modeled by fourth-order biharmonic wave equations, holds broad engineering applications. These applications span diverse fields, including the design of ultra-broadband elastic cloaking \cite{DZAL-PRL-2018, FGE-PRL-2009, FGE-JCP-2011}, platonic crystals \cite{S-UA-2013, H-UL-2014, HCMM-RRSA-2016}, and the exploration of acoustic black hole concepts \cite{PGCS-JSV-2020}. Consequently, ongoing research in theoretical analysis, numerical simulations, and industrial manufacturing continues to draw considerable attention from both engineering and mathematical communities.

Most works in the literature focus on the static problem in a bounded domain, which is formulated by the bi-Laplacian equation. When addressing the fourth-order problem using the finite element method, standard $H^2$-conforming methods necessitate $C^1$-continuous piecewise polynomials on the mesh, a challenge in practical implementation. Alternatively, various nonconforming and discontinuous finite element methods have emerged, such as the weak Galerkin finite element methods supplemented with stabilizers \cite{MWZ-JSC-2014, YZ-SIAM-2020, ZZ-JSC-2015}; the virtual element method, which requires no global $C^0$ regularity for the numerical solution \cite{AMV-MMMAS-2018, ZCZ-MMMAS-2016}; and the mixed element method, effectively reducing the fourth-order problem to coupled second-order problems \cite{AD-NM-2001, C-FEM-1978, CR-1974, CG-CMAME-1975}. These methods have undergone comprehensive analysis. 

When compared with the results concerning the bi-Laplacian equation, the findings are relatively limited for the biharmonic wave scattering problems in unbounded domains. In \cite{DL-2023-arXiv}, the initial theoretical analysis of the boundary integral equation method was provided for solving the biharmonic wave equation. Through the introduction of two auxiliary variables, the biharmonic wave equation was split into the Helmholtz and modified Helmholtz equations. Subsequently, the Holmholtz and modified Helmholtz wave components were represented using the double- and single-layer potentials. The well-posedness of the coupled boundary integral system was established by applying the Riesz--Fredholm theory. If the exterior problem is approached using the variational approach with transparent boundary conditions (TBCs), the studies concerning waveguide and obstacle scattering problems were presented in \cite{BCF-CMS-2019, BH-SIAM-2020} under various boundary conditions, including clamped, simply supported, roller-supported, or free plate boundary conditions. Numerically, a mixed element method was proposed in \cite{YLYZ-RAM-2023, YL-JCP-2023} by introducing two auxiliary variables and  decomposing the biharmonic problem into the Helmholtz and modified Helmholtz equations. Subsequently, TBCs were introduced for each equation. Particularly, the linear finite element method, incorporating interior penalty and boundary penalty, was proposed in \cite{YL-JCP-2023} to effectively reduce the oscillation of the bending moment.

The method of perfectly matched layer (PML) is a widely utilized domain truncation technique. In contrast to the nonlocal TBC method, the PML method generates a local boundary condition on the outer surface of the layer by integrating an artificial absorbing region around the domain of interest. The ease of handling the local boundary condition has contributed to the widespread adoption of this method ever since its inception by B\'erenger \cite{B-JCP-1994} for solving the time-dependent Maxwell equations. It has found extensive applications in solving various wave scattering problems, including, for example, acoustic waves \cite{CL-SINUM-2005}, electromagnetic waves \cite{BW-SINUM, BP-MC-2007, LWZ-SIMA-2011}, and elastic waves \cite{CXZ-MC-2016, BPT-MC-2010}. The PML method has also been utilized numerically in solving biharmonic wave scattering problems \cite{SM-JSV-2011, MB-IJSS-2018, MB-IJSS-2016}, highlighting its convenience and accuracy. However, to our knowledge, a comprehensive discussion regarding the well-posedness of the PML method and its convergence has not been documented in existing literature. This paper aims to address these gaps.

In this paper, we investigate the scattering of flexural waves resulting from a plane incident wave interacting with a one-dimensional periodic array of cavities within an infinite elastic thin plate. The wave propagation is described by the fourth-order biharmonic wave equation. Because of the periodic characteristics of both the incident wave and the cavities, the solution complies with quasi-periodic conditions, allowing us to formulate the problem within a single periodic cell. The TBCs are derived by incorporating the bounded outgoing wave condition, utilizing the Fourier series expansion of the solution in regions distant from the cavities. With the aid of the TBCs, the scattering problem is equivalently transformed from an unbounded domain to a bounded one. The corresponding variational problem is shown to satisfy G\r{a}rding's inequality, and its well-posedness is established through the utilization of the Fredholm alternative theorem. To replace the nonlocal TBCs, the PML method is adopted through the complex coordinate stretching scheme \cite{TC-IEEE-1997}. Alternatively, the unbounded domain is truncated by imposing homogeneous boundary conditions on the wave field and its normal derivative at the outer boundary of the PML region. Upon studying the Fourier series expansion of the solution to the PML problem, we deduce equivalent TBCs to reformulate the PML problem in the domain where the original scattering problem, along with the TBC, is imposed. The well-posedness of the PML problem is confirmed through an examination of its variational formulation. Additionally, the PML solution demonstrates exponential convergence concerning the thickness of the PML regions towards the solution of the original scattering problem. For a comprehensive account of related electromagnetic wave scattering problems in periodic structures, we reference \cite{BL-2022}.

The paper is outlined as follows. Section \ref{Section2} introduces the model equations. The TBCs are derived in Section \ref{Section3}. Section \ref{Section4} details the reduction of the scattering problem to a bounded domain using the TBCs, along with the discussion on the well-posedness of the variational problem. Section \ref{Section5} addresses the PML problem, including investigations into its well-posedness and convergence. Finally, the paper concludes with general remarks in Section \ref{Section6}.

\section{Problem formulation}\label{Section2}

Let us examine the scattering phenomenon of an incident wave interacting with a one-dimensional periodic array of cavities in an infinitely extending elastic thin plate, which is characterized by the Kirchhoff--Love model and is depicted in Figure \ref{pg}. Assume that the alignment of the cavities coincides with the $x_1$-axis, exhibiting a periodicity of $\Lambda$. Consider an incident field represented as a time-harmonic plane wave given by
\[
u^i(x)=e^{{\rm i}(\alpha x_1-\beta x_2)}, \quad x\in\mathbb R^2,
\]
where $\alpha = \kappa\sin\theta, \beta = \kappa\cos\theta$ with $\kappa > 0$ and $\theta \in \left(-\frac{\pi}{2}, \frac{\pi}{2}\right)$ denoting the wavenumber and the incident angle, respectively. It can be verified that the incident field $u^i$ satisfies the biharmonic wave equation 
\[
 \Delta^2 u^i - \kappa^4 u^i=0\quad\text{in} ~ \mathbb R^2. 
\]

Due to the periodic nature of the geometry, the problem can be confined to a single periodic cell. Denote by $\Omega_c$ the cavity with a Lipschitz continuous boundary $\Gamma_c$. Let $R$ be a rectangular domain that is sufficiently large to enclose the region $\Omega_c$. Without loss of generality, let $R$ be defined as $R=\left\{x\in\mathbb{R}^2: 0<x_1<\Lambda, h_2<x_2<h_1\right\}$, where $h_k, k=1, 2$ are constants. Additionally, define $\Gamma_k=\left\{x\in\mathbb{R}^2: 0<x_1<\Lambda, x_2=h_k\right\}$ for $k=1, 2$, $\Gamma_l=\left\{x\in\mathbb{R}^2: x_1=0, h_2<x_2<h_1\right\}$, and $\Gamma_r=\left\{x\in\mathbb{R}^2: x_1=\Lambda, h_2<x_2<h_1\right\}$. Let $\Omega=R\setminus\overline{\Omega_c}$. Define $\Omega_1$ and $\Omega_2$ as the regions above and below $\Gamma_1$ and $\Gamma_2$, respectively.

\begin{figure}
\centering
\includegraphics[width=0.6\textwidth]{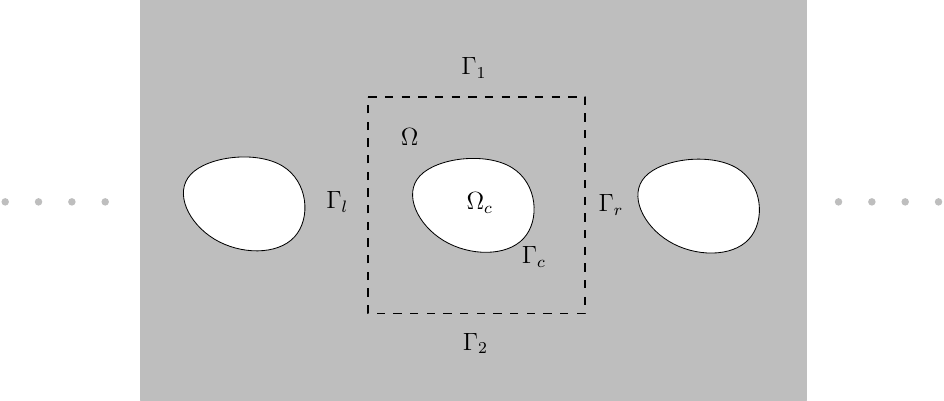}
\caption{Schematic of the problem geometry.}
\label{pg}
\end{figure}

The out-of-plane displacement of the plate, denoted as $u$ and referred to as the total field, also satisfies the biharmonic wave equation
\begin{equation}\label{TotalBiharmonic}
 \Delta^2 u-\kappa^4 u=0\quad {\rm in} ~ \Omega. 
\end{equation}
The total field is assumed to satisfy the Dirichlet boundary condition, known as the clamped boundary condition, on $\Gamma_c$: 
\begin{equation}\label{cbc}
 u=0,\quad \partial_{\nu} u=0,
\end{equation}
where $\nu$ denotes the unit normal vector on $\Gamma_c$. It is worth noting that other types of boundary conditions, such as the Neumann boundary condition, often known as the free plate boundary condition, can be similarly taken into account.

Given the periodic nature of both the structure and the incident wave, the solution to \eqref{TotalBiharmonic}--\eqref{cbc} demonstrates quasi-periodicity. Specifically, if $u$ is a solution to \eqref{TotalBiharmonic}--\eqref{cbc}, then $u(x) e^{-{\rm i}\alpha x_1}$ is a periodic function of $x_1$ with a period of $\Lambda$. This characteristic gives rise to the quasi-periodic boundary condition on $\Gamma_l$ and $\Gamma_r$, i.e., $u$ satisfies $u(0, x_2) = e^{-{\rm i}\alpha\Lambda}u(\Lambda, x_2)$. Furthermore, to ensure the well-posedness of the problem, it is essential to impose a bounded outgoing wave condition on the scattered field $u^s=u-u^i$ in $\Omega_1$ and the total field $u$ in $\Omega_2$.

We introduce notations and function spaces employed in this work. Denote by $H^2(\Omega)$ the standard Sobolev space,  comprising functions with square-integrable values, as well as square-integrable first and second partial derivatives. Let us define the quasi-periodic function space 
\[
H_{\rm qp}^2(\Omega)=\left\{u\in H^2(\Omega): u(\Lambda, x_2)=u(0, x_2)e^{{\rm i}\alpha\Lambda}\right\},
\]
along with its subspace
\[
H_{{\rm qp}, \Gamma_c}^{2}(\Omega)=\left\{u\in H^2_{\rm qp}(\Omega): u=0, \partial_{\nu} u=0\text{ on }\Gamma_c\right\}.
\]
Clearly, $H_{\rm qp}^2(\Omega)$ and $H_{{\rm qp}, \Gamma_c}^{2}(\Omega)$ are subspaces of $H^2(\Omega)$ equipped with the standard $H^2$-norm.

Given a function $u\in H_{\rm qp}^2(\Omega)$, it allows for a Fourier expansion on $\Gamma_k, k=1, 2:$
\[
u(x, h_k)=\sum\limits_{n\in\mathbb{Z}} u^{(n)}(h_k) e^{{\rm i}\alpha_n x_1},
\]
where
\[
 \alpha_n=\alpha+n\left(\frac{2\pi}{\Lambda}\right), \quad u^{(n)}(h_k)=\frac{1}{\Lambda}\int_{0}^{\Lambda} u(x, h_k) e^{-{\rm i}\alpha_n x_1}{\rm d}x_1. 
\]
The trace function space $H^s(\Gamma_k)$, where $s \in \mathbb{R}$, is defined as follows:
\[
H^{s}(\Gamma_k)=\left\{u\in L^2(\Gamma_k): \|u\|_{H^s(\Gamma_k)}<\infty\right\},
\]
with the norm given by
\[
\|u\|_{H^s(\Gamma_k)}=\left(\Lambda\sum\limits_{n\in\mathbb{Z}}\left(1+\alpha_n^2\right)^s |u^{(n)}(h_k)|^2\right)^{1/2}.
\]

In this paper, whenever $a\lesssim b$ is used, it denotes $a\leq Cb$, with $C$ representing a positive constant. In this context, the values of the constants $c_j$ are positive and may vary in different steps of the proof. Although the specific values of $C$ and $c_j$ are not explicitly stated, their dependence should be apparent from the context.

\section{Transparent boundary conditions}\label{Section3}

In this section, we address the challenge posed by formulating the problem in an unbounded domain. To overcome this obstacle, we propose introducing an equivalent transparent boundary condition (TBC) on $\Gamma_k, k=1, 2$ with the objective of transforming the problem into the bounded domain $\Omega$. 

Let $\nu=(\nu_1, \nu_2)$ and $\tau=(\tau_1, \tau_2)$ be the unit normal and tangent vectors, respectively, to the boundary of $\Omega$. Clearly, we have $\tau_1=-\nu_2$ and $\tau_2=\nu_1$. Define the normal and tangential derivatives
\[
\partial_\nu:=\nu_1 \partial_{x_1}+\nu_2\partial_{x_2},\quad
\partial_\tau:=-\nu_2\partial_{x_1}+\nu_1\partial_{x_2}.
\]
For $0\leq\mu<1$, referred to as the Poisson ratio, define the surface differential operators (cf. \cite{HW-2021}):
\begin{equation}\label{MNOperator}
Mu:=\mu\Delta u+(1-\mu)M_0 u,\quad 
Nu:=-\partial_\nu \Delta u-(1-\mu)\partial_\tau N_0 u,
\end{equation}
where $M_0$ and $N_0$ are explicitly given by
\[
\left\{
\begin{aligned}
&M_0 u:=\nu_1^2\frac{\partial^2 u}{\partial x_1^2}+2\nu_1\nu_2\frac{\partial^2 u}{\partial x_1\partial x_2}+
	\nu_2^2\frac{\partial^2 u}{\partial x_2^2},\\
&N_0 u:=\left(\nu_1^2-\nu_2^2\right)	\frac{\partial^2 u}{\partial x_1\partial x_2}
	-\nu_1\nu_2\left(\frac{\partial^2 u}{\partial x_1^2}-\frac{\partial^2 u}{\partial x_2^2}\right).
\end{aligned}
\right.
\]

First, we derive the TBC on $\Gamma_1$. Based on the bounded outgoing wave condition, it is shown in \cite{YLYZ-RAM-2023}  that the scattered field $u^s$ can be represented by a Fourier series expansion in the domain $\Omega_1$:
\begin{equation}\label{ExpansionUs}
	u^s(x_1, x_2)=\sum\limits_{n\in\mathbb{Z}} H_1^{(n)}e^{{\rm i}\alpha_n x_1+{\rm i}\beta_n(x_2-h_1)}
	+\sum\limits_{n\in\mathbb{Z}} U_1^{(n)}e^{{\rm i}\alpha_n x_1-\gamma_n(x_2-h_1)},
\end{equation}
where $H_1^{(n)}, U_1^{(n)}\in\mathbb C$ are the Fourier coefficients, and 
\begin{equation}\label{alphabeta}
	\beta_n=\left\{
	\begin{aligned}
	&(\kappa^2-\alpha_n^2)^{1/2}\quad & &\text{if} ~ \kappa>|\alpha_n|,\\
	&{\rm i}(\alpha_n^2-\kappa^2)^{1/2}\quad& &\text{if} ~ \kappa<|\alpha_n|,
	\end{aligned}
	\right. \quad
	\gamma_n=(\kappa^2+\alpha_n^2)^{1/2}.
\end{equation}
Here, we assume that $\beta_n\neq 0$ for all $n\in\mathbb Z$ to rule out the occurrence of resonances. 

Let $\left(u, \partial_{\nu}u\right)=(f_1, g_1)$ and $\left(u^s, \partial_{\nu}u^s\right)=(\hat{f}_1, \hat{g}_1)$
be the Dirichlet data for the total and scattered fields on $\Gamma_1$, respectively. It is clear to note that these data satisfy the relations
\[
	f_1(x_1)=\hat{f}_1(x_1)+e^{{\rm i}\left(\alpha x_1-\beta h_1\right)},\quad 
	g_1(x_1)=\hat{g}_1(x_1)-{\rm i}\beta e^{{\rm i}\left(\alpha x_1-\beta h_1\right)}.
\]
Being quasi-periodic functions, $\hat f_1$ and $\hat g_1$ admit the Fourier series expansions
\[
u^s(x_1, h)=\hat f_1(x_1) =\sum_{n\in\mathbb Z} \hat f_1^{(n)}e^{{\rm i}\alpha_n x_1},\quad \partial_\nu u^s(x_1, h_1)=\hat g_1(x_1) =\sum_{n\in\mathbb Z} \hat g_1^{(n)}e^{{\rm i}\alpha_n x_1},
\]
where $\hat f_1^{(n)}, \hat g_1^{(n)}\in\mathbb C$ are the Fourier coefficients. 

On the other hand, evaluating the scattered field $u^s$, as defined in \eqref{ExpansionUs}, and its normal derivative $\partial_{x_2} u^s$ on $\Gamma_1$, we obtain 
\[
\left\{
\begin{aligned}
	u^s(x_1,  h_1)&=\sum\limits_{n\in\mathbb{Z}} H_1^{(n)}e^{{\rm i}\alpha_n x_1}
	+\sum\limits_{n\in\mathbb{Z}} U_1^{(n)}e^{{\rm i}\alpha_n x_1},\\
	\partial_{x_2} u^s(x_1, h_1)&=\sum\limits_{n\in\mathbb{Z}} {\rm i}\beta_n H_1^{(n)}
		e^{{\rm i}\alpha_n x_1}-\sum\limits_{n\in\mathbb{Z}} \gamma_n U_1^{(n)}e^{{\rm i}\alpha_n x_1}.
\end{aligned}
\right.
\]
Combining the above equations, we have from straightforward calculations that the scattered field $u^s$ in domain $\Omega_1$ can be expressed as
\begin{align}\label{Uupper}
	u^s(x_1, x_2)&=\sum\limits_{n\in\mathbb{Z}} \left(
		\frac{\gamma_n \hat{f}^{(n)}_1+\hat{g}^{(n)}_1}{\gamma_n+{\rm i}\beta_n}\right)e^{{\rm i}\alpha_n x_1
		+{\rm i}\beta_n(x_2-h_1)}\notag\\
		&\quad +\sum\limits_{n\in\mathbb{Z}}\left( \frac{{\rm i}\beta_n \hat{f}^{(n)}_1-\hat{g}^{(n)}_1}{\gamma_n+{\rm i}\beta_n}\right)e^{{\rm i}\alpha_n x_1-\gamma_n(x_2-h_1)}.
\end{align}

On $\Gamma_1$, the surface differential operators $M$ and $N$ given in \eqref{MNOperator} can be simplified to 
\begin{equation}\label{MNupper}
	N_1 u=-(2-\mu)\frac{\partial^3 u}{\partial x_1^2\partial x_2}-\frac{\partial^3 u}{\partial x_2^3},\quad
	M_1 u=\mu\frac{\partial^2 u}{\partial x_1^2}+\frac{\partial^2 u}{\partial x_2^2}. 
\end{equation}
Substituting \eqref{Uupper} into \eqref{MNupper} yields the TBC of the scattered field $u^s$ on $\Gamma_1$:
 \begin{equation*}
N_1 u^s=T_{11}^{(1)} \hat{f}_1+T_{12}^{(1)}\hat{g}_1,\quad 
M_1 u^s=T_{21}^{(1)}\hat{f}_1+T_{22}^{(1)}\hat{g}_1.
\end{equation*}
Here, the Dirichlet-to-Neumann (DtN) operators $T_{ij}^{(1)}, i,j=1, 2$ are given by 
\begin{equation}\label{T1}
\left\{
\begin{aligned}
	(T_{11}^{(1)} f)(x_1) &=\sum\limits_{n\in\mathbb{Z}}
	{\rm i}\beta_n\gamma_n\left(\gamma_n-{\rm i}\beta_n\right)f^{(n)} e^{{\rm i}\alpha_n x_1},\\
	(T_{21}^{(1)} f)(x_1) &=-\sum\limits_{n\in\mathbb{Z}} \left(\mu\alpha_n^2
	-{\rm i}\beta_n\gamma_n\right) f^{(n)} e^{{\rm i}\alpha_n x_1},\\
	(T_{12}^{(1)}g)(x_1) &=-\sum\limits_{n\in\mathbb{Z}} \left(\mu\alpha_n^2-{\rm i}\beta_n\gamma_n\right) g^{(n)} 
	e^{{\rm i}\alpha_n x_1},\\
	(T_{22}^{(1)}g)(x_1) &=-\sum\limits_{n\in\mathbb{Z}} \left(\gamma_n-{\rm i}\beta_n\right)g^{(n)} e^{{\rm i}\alpha_n x_1},
\end{aligned}
\right.
\end{equation}
where $f^{(n)}$ and $g^{(n)}$ are the Fourier coefficients of $f$ and $g$, respectively. Noting $u=u^s+u^i$, we deduce the TBC for the total field $u$ on $\Gamma_1$:  
 \begin{equation}\label{DtNGamma1}
N_1 u=T_{11}^{(1)} f_1+T_{12}^{(1)}g_1+p_1,\quad 
M_1 u=T_{21}^{(1)}f_1+T_{22}^{(1)}g_1+p_2.
\end{equation}
where
\begin{equation}\label{p12}
p_1(x_1)=-\left(2{\rm i}\beta\alpha^2+2\beta^2\gamma\right) e^{{\rm i}(\alpha x_1-\beta h_1)},\quad
p_2(x_1)=-(2\beta^2+2{\rm i}\beta\gamma)e^{{\rm i}(\alpha x_1-\beta h_1)}.
\end{equation}

Given the similarity in the derivation process of the TBC to that of $\Gamma_2$, we provide a brief overview of the procedure and present the resulting TBC. In accordance with the bounded outgoing wave condition, the total field $u$ exhibits the Fourier series expansion in $\Omega_2$: 
\begin{equation}\label{ExpansionUGamma2}
	u(x_1, x_2)=\sum\limits_{n\in\mathbb{Z}} H_2^{(n)}e^{{\rm i}\alpha_n x_1-{\rm i}\beta_n(x_2-h_2)}
	+\sum\limits_{n\in\mathbb{Z}} U_2^{(n)}e^{{\rm i}\alpha_n x_1+\gamma_n(x_2-h_2)}. 
\end{equation}
Evaluating \eqref{ExpansionUGamma2} and its normal derivative on $\Gamma_2$, we obtain 
\begin{equation}\label{EquationGamma2}
\left\{
\begin{aligned}
u(x_1, h_2)&=\sum\limits_{n\in\mathbb{Z}} H_2^{(n)}e^{{\rm i}\alpha_n x_1}
	+\sum\limits_{n\in\mathbb{Z}} U_2^{(n)}e^{{\rm i}\alpha_n x_1}
	=\sum\limits_{n\in\mathbb{Z}}f^{(n)}_2 e^{{\rm i}\alpha_n x_1},\\
\partial_{\nu} u(x_1, h_2)&=\sum\limits_{n\in\mathbb{Z}} {\rm i}\beta_n H_2^{(n)}e^{{\rm i}\alpha_n x_1}
	-\sum\limits_{n\in\mathbb{Z}} \gamma_n U_2^{(n)}e^{{\rm i}\alpha_n x_1}
	=\sum\limits_{n\in\mathbb{Z}}g^{(n)}_2 e^{{\rm i}\alpha_n x_1},
\end{aligned}
\right.
\end{equation}
where $(u, \partial_\nu u)=(f_2, g_2)$ are the Dirichlet data on $\Gamma_2$ and have the Fourier series expansions
\[
	f_2(x_1)=\sum\limits_{n\in\mathbb{Z}}f^{(n)}_2 e^{{\rm i}\alpha_n x_1},\quad
	g_2(x_1)=\sum\limits_{n\in\mathbb{Z}}g^{(n)}_2 e^{{\rm i}\alpha_n x_1}.
\]

By solving the system \eqref{EquationGamma2}, we deduce that the total field $u$ in $\Omega_2$ admits the Fourier series expansion
\begin{align}\label{Ulower}
	u(x_1, x_2)&=\sum\limits_{n\in\mathbb{Z}} \left(
		\frac{\gamma_n f^{(n)}_2+g^{(n)}_2}{\gamma_n+{\rm i}\beta_n}\right)e^{{\rm i}\alpha_n x_1-{\rm i}\beta_n(x_2-h_2)}\notag\\
		&\quad +\sum\limits_{n\in\mathbb{Z}} \left(\frac{{\rm i}\beta_n f^{(n)}_2-g^{(n)}_2}{\gamma_n+{\rm i}\beta_n}
		\right)e^{{\rm i}\alpha_n x_1+\gamma_n(x_2-h_2)}.
\end{align}
Noting that the surface differential operator $M$ and $N$ on $\Gamma_2$ can be simplified to 
\begin{equation}\label{MNlower}
	N_2 u=(2-\mu)\frac{\partial^3 u}{\partial x_1^2\partial x_2}+\frac{\partial^3 u}{\partial x_2^3},\quad
	M_2 u=\mu\frac{\partial^2 u}{\partial x_1^2}+\frac{\partial^2 u}{\partial x_2^2}, 
\end{equation}
we substitute \eqref{Ulower} into \eqref{MNlower} and obtain the TBC of the total field $u$ on $\Gamma_2$: 
\begin{equation}\label{DtNGamma2}
N_2 u=T_{11}^{(2)} f_2+T_{12}^{(2)}g_2,\quad  M_2 u=T_{21}^{(2)}f_2+T_{22}^{(2)}g_2,
\end{equation}
where the DTN operators $T_{ij}^{(2)}, i,j=1,2$ are defined as
\begin{equation*}
\left\{
\begin{aligned}
	(T_{11}^{(2)} f)(x_1) &=\sum\limits_{n\in\mathbb{Z}}
	{\rm i}\beta_n\gamma_n\left(\gamma_n-{\rm i}\beta_n\right)f^{(n)}_2e^{{\rm i}\alpha_n x_1},\\
	(T_{21}^{(2)} f)(x_1) &=-\sum\limits_{n\in\mathbb{Z}} \left(\mu\alpha_n^2
	-{\rm i}\beta_n\gamma_n\right) f^{(n)}_2e^{{\rm i}\alpha_n x_1},\\
	(T_{12}^{(2)}g)(x_1) &=-\sum\limits_{n\in\mathbb{Z}} \left(\mu\alpha_n^2-{\rm i}\beta_n\gamma_n\right) g^{(n)}_2
	e^{{\rm i}\alpha_n x_1},\\
	(T_{22}^{(2)}g)(x_1) &=-\sum\limits_{n\in\mathbb{Z}} \left(\gamma_n-{\rm i}\beta_n\right)g^{(n)}_2e^{{\rm i}\alpha_n x_1}. 
\end{aligned}
\right.
\end{equation*}

The following result concerns the properties of the DtN operators $T^{(1)}_{ij}$ and $T^{(2)}_{ij}$, where $i,j=1,2.$

\begin{lemma}\label{BoundTBC}
For $k=1,2$, the DtN operators $T^{(k)}_{11}: H^{3/2}(\Gamma_k)\to H^{-3/2}(\Gamma_k)$,
$T^{(k)}_{12}: H^{1/2}(\Gamma_k)\to H^{-3/2}(\Gamma_k)$, $T^{(k)}_{21}: H^{3/2}(\Gamma_k)\to H^{-1/2}(\Gamma_k)$, and $T^{(k)}_{22}: H^{1/2}(\Gamma_k)\to H^{-1/2}(\Gamma_k)$ are bounded. 
\end{lemma}

\begin{proof}
We only prove the results for the operators $T_{ij}^{(1)}$, as the corresponding properties for the operators $T_{ij}^{(2)}$ can be obtained in the same manner. It is clear to note from \eqref{alphabeta} that
\begin{equation}\label{abg}
\lim\limits_{|n|\rightarrow\infty} \frac{|\beta_n|}{|\alpha_n|}=1,\quad
\lim\limits_{|n|\rightarrow\infty} \frac{|\gamma_n|}{|\alpha_n|}=1.
\end{equation}
For a given $f\in H^{3/2}(\Gamma_1)$, we have from \eqref{T1} and \eqref{abg} that
\begin{align*}
	\|T^{(1)}_{11} f\|^2_{H^{-3/2}(\Gamma_1)} &=\Lambda
	\sum\limits_{n\in\mathbb{Z}}\left(1+\alpha_n^2\right)^{-3/2}\big|
	{\rm i}\beta_n\gamma_n\left(\gamma_n-{\rm i}\beta_n\right)f^{(n)}\big|^2\\
	&\lesssim \sum\limits_{n\in\mathbb{Z}}(1+\alpha_n^2)^{3/2}|f^{(n)}|^2\lesssim \|f\|^2_{H^{3/2}(\Gamma_1)},
\end{align*}
and
\begin{align*}
	\|T^{(1)}_{21} f\|^2_{H^{-1/2}(\Gamma_1)} &=\Lambda
	\sum\limits_{n\in\mathbb{Z}}\left(1+\alpha_n^2\right)^{-1/2}\big|
	 \left(\mu\alpha_n^2-{\rm i}\beta_n\gamma_n\right)f^{(n)}\big|^2\\
	&\lesssim \sum\limits_{n\in\mathbb{Z}}(1+\alpha_n^2)^{3/2}|f^{(n)}|^2\lesssim \|f\|^2_{H^{3/2}(\Gamma_1)}.
\end{align*}
Similarly, for any function $f\in H^{1/2}(\Gamma_1)$, we deduce from \eqref{T1} and \eqref{abg} that
\begin{align*}
	\|T^{(1)}_{12} f\|^2_{H^{-3/2}(\Gamma_1)} &=\Lambda
	\sum\limits_{n\in\mathbb{Z}}\left(1+\alpha_n^2\right)^{-3/2}\big|
	 \left(\mu\alpha_n^2-{\rm i}\beta_n\gamma_n\right)g^{(n)}\big|^2\\
	&\lesssim \sum\limits_{n\in\mathbb{Z}}(1+\alpha_n^2)^{1/2}|f^{(n)}|^2\lesssim \|f\|^2_{H^{1/2}(\Gamma_1)},
\end{align*}
and
\begin{align*}
	\|T^{(1)}_{22} f\|^2_{H^{-1/2}(\Gamma_1)} &=\Lambda
	\sum\limits_{n\in\mathbb{Z}}\left(1+\alpha_n^2\right)^{-1/2}\big|
	 \left(\gamma_n-{\rm i}\beta_n\right)f^{(n)}\big|^2\\
	&\lesssim \sum\limits_{n\in\mathbb{Z}}(1+\alpha_n^2)^{1/2}|f^{(n)}|^2\lesssim \|f\|^2_{H^{1/2}(\Gamma_1)},
\end{align*}
thus completing the proof.
\end{proof}

\begin{lemma}\label{PositiveDtN}
If $|n|$ is sufficiently large, then the following inequality holds for any complex values $f$ and $g$:
\begin{eqnarray*}
	\Re\big\{
	-{\rm i}\beta_n\gamma_n\left(\gamma_n-{\rm i}\beta_n\right)|f|^2
	+\left(\mu\alpha_n^2-{\rm i}\beta_n\gamma_n\right)g\bar{f}\\
	+\left(\mu\alpha_n^2-{\rm i}\beta_n\gamma_n\right) f\bar{g}
	+ \left(\gamma_n-{\rm i}\beta_n\right)|g|^2
	\big\}\geq0.
\end{eqnarray*}
\end{lemma}

\begin{proof}
By the definitions of $\gamma_n$ and $\beta_n$ given in \eqref{alphabeta}, a straightforward calculation shows that for sufficiently large $|n|$
\[
-{\rm i}\beta_n\gamma_n\left(\gamma_n-{\rm i}\beta_n\right)>0,\quad\mu\alpha_n^2-{\rm i}\beta_n\gamma_n>0,\quad
\gamma_n-{\rm i}\beta_n>0.
\]
It suffices to demonstrate for sufficiently large $|n|$ that
\[
	-{\rm i}\beta_n\gamma_n\left(\gamma_n-{\rm i}\beta_n\right)|f|^2
	-2\left(\mu\alpha_n^2-{\rm i}\beta_n\gamma_n\right)|f||g|
	+\left(\gamma_n-{\rm i}\beta_n\right)|g|^2\geq0.
\]
Utilizing the Cauchy inequality, we deduce from a simple calculation that
\begin{align*}
	&\big((-{\rm i}\beta_n\gamma_n)^{1/2}\left(\gamma_n-{\rm i}\beta_n\right)\big)^2
		-\big( \mu\alpha_n^2-{\rm i}\beta_n\gamma_n\big)^2\\
	&= -\gamma_n^2\beta_n^2-\mu\alpha_n^4-2{\rm i}(1-\mu)\gamma_n\beta_n\alpha_n^2\\
	&= \alpha_n^4-\kappa^4-\mu\alpha_n^4-2{\rm i}(1-\mu)\gamma_n\beta_n\alpha_n^2\\
	&= (1-\mu)\left(\alpha_n^4-2{\rm i}\gamma_n\beta_n\alpha_n^2\right)- \kappa^4,
\end{align*}
which is positive for sufficiently large $|n|$ by noting that $\mu<1$ and $\alpha_n^4-2{\rm i}\gamma_n\beta_n\alpha_n^2\to\infty$ as $|n|\to\infty$.
\end{proof}

\begin{lemma}\label{PositiveTBC}
For any $f\in H^{3/2}(\Gamma_k)$ and $g\in H^{1/2}(\Gamma_k)$ with $k=1, 2$, there exists a positive constant $C$ such that 
\begin{eqnarray*}
-\Re\left\{\int_{\Gamma_k} \Big(T_{11}^{(k)} |f|^2+T_{12}^{(k)} g\bar{f}
 +T_{21}^{(k)} f\bar{g}+T_{22}^{(k)} |g|^2\Big){\rm d}s\right\}\\
 \geq -C\Big(\|f\|^2_{L^{2}(\Gamma_k)}+\|g\|^2_{L^2(\Gamma_k)}\Big). 
 \end{eqnarray*}
\end{lemma}

\begin{proof}
We prove the result for $k=1$, as the result for $k=2$ can be obtained similarly. By the definitions of $T^{(1)}_{ij}$ and Lemma \ref{PositiveDtN}, we have
\begin{align*}
&-\Re\int_{\Gamma_1} \left(T_{11}^{(1)} |f|^2 +T_{12}^{(1)} g\bar{f}
 +T_{21}^{(1)} f\bar{g}+T_{22}^{(1)}|g|^2\right){\rm d}s\\
 &=\Lambda\sum\limits_{|n|\leq n_0}
 	\Re\Big\{
	-{\rm i}\beta_n\gamma_n\left(\gamma_n-{\rm i}\beta_n\right)|f^{(n)}|^2
	+\left(\mu\alpha_n^2-{\rm i}\beta_n\gamma_n\right)g^{(n)}\overline{f^{(n)}}\\
&\quad +\left(\mu\alpha_n^2-{\rm i}\beta_n\gamma_n\right) f^{(n)} \overline{g^{(n)}}
	+ \left(\gamma_n-{\rm i}\beta_n\right)|g^{(n)}|^2
	\Big\}\\
&\quad +\Lambda\sum\limits_{|n|> n_0}	\Re\Big\{
	-{\rm i}\beta_n\gamma_n\left(\gamma_n-{\rm i}\beta_n\right)|f^{(n)}|^2\\
&\quad +\left(\mu\alpha_n^2-{\rm i}\beta_n\gamma_n\right)g^{(n)}\overline{f^{(n)}}
	+\left(\mu\alpha_n^2-{\rm i}\beta_n\gamma_n\right) f^{(n)}\overline{g^{(n)}}
	+ \left(\gamma_n-{\rm i}\beta_n\right)|g^{(n)}|^2
	\Big\}\\
&\geq \Lambda\sum\limits_{|n|\leq n_0}
 	\Re\Big\{
	-{\rm i}\beta_n\gamma_n\left(\gamma_n-{\rm i}\beta_n\right)|f^{(n)}|^2
	+\left(\mu\alpha_n^2-{\rm i}\beta_n\gamma_n\right)g^{(n)}\overline{f^{(n)}}\\
&\quad +\left(\mu\alpha_n^2-{\rm i}\beta_n\gamma_n\right) f^{(n)}\overline{g^{(n)}}
	+ \left(\gamma_n-{\rm i}\beta_n\right)|g^{(n)}|^2
	\Big\}\\
&\geq -C\left(\|f\|^2_{L^{2}(\Gamma_1)}+\|g\|^2_{L^2(\Gamma_1)}\right),
\end{align*}
where $n_0$ is the smallest integer such that Lemma \ref{PositiveDtN} holds. 
\end{proof}

Utilizing the TBCs given by \eqref{DtNGamma1} and \eqref{DtNGamma2}, we transform the original problem
\eqref{TotalBiharmonic}--\eqref{cbc} from an unbounded domain into the bounded domain $\Omega$, which is to find a quasi-periodic function $u$ satisfying
\begin{equation}\label{TotalBiharmonicReduced}
\left\{
\begin{aligned}
	& \Delta^2 u-\kappa^4 u=0 &&\text{in} ~ \Omega,\\
	& u=0,\quad \partial_{\nu} u=0 & &\text{on} ~ \Gamma_c,\\
	& N_1 u=T_{11}^{(1)} f_1+T_{12}^{(1)}g_1+p_1 &&\text{on} ~ \Gamma_1,\\
	& M_1 u=T_{21}^{(1)}f_1+T_{22}^{(1)}g_1+p_2 &&\text{on} ~ \Gamma_1,\\
	& N_2 u=T_{11}^{(2)} f_2+T_{12}^{(2)}g_2 & &\text{on} ~ \Gamma_2,\\
	& M_2 u=T_{21}^{(2)}f_2+T_{22}^{(2)}g_2 & &\text{on} ~ \Gamma_2,
\end{aligned}
\right.
\end{equation}
where $p_1$ and $p_2$ are defined in \eqref{p12}, and $(f_k, g_k)$ are the Dirichlet data of the total field $u$ on $\Gamma_k$ for $k=1, 2.$ The objective of this study is to examine the PML formulation applied to the boundary value problem \eqref{TotalBiharmonicReduced} and to establish the convergence of the PML solution.

\section{The variational problem}\label{Section4}

In this section, we present a variational formulation for the boundary value problem \eqref{TotalBiharmonicReduced} and examine its well-posedness.

Observing that the bi-Laplacian can be expressed in terms of the Poisson ratio, as demonstrated in \cite{HW-2021}, we have
\begin{align*}
 \Delta^2 u -\kappa^4 u & = \frac{\partial^2 }{\partial x_1^2}\left(\frac{\partial^2 u}{\partial x_1^2}+\mu\frac{\partial^2 u}{\partial x_2^2} \right)+2(1-\mu)\frac{\partial^2 }{\partial x_1 \partial x_2}\left(\frac{\partial^2 u}{\partial x_1\partial x_2}\right)\\
&\quad + \frac{\partial^2 }{\partial x_2^2}\left(\frac{\partial^2 u}{\partial x_2^2}+\mu\frac{\partial^2 u}{\partial x_1^2} \right)-\kappa^4 u. 
\end{align*}
Multipling both sides of the above equation with a test function $v\in H_{{\rm qp}, \Gamma_c}^2(\Omega)$, integrating across the domain $\Omega$, and applying integration by parts, we obtain 
\begin{align}\label{Identity}
\int_{\Omega} \bigg[\mu\Delta u\Delta \bar{v}+(1-\mu)\sum\limits_{i,j=1}^2\frac{\partial^2 u}{\partial x_i\partial x_j}\frac{\partial^2 \bar{v}}{\partial x_i\partial x_j}-\kappa^4 u\bar{v}\bigg]{\rm d}x\notag\\
-\sum_{k=1}^2\int_{\Gamma_k} \left(\bar{v}N_k u+\partial_\nu \bar{v}M_k u\right){\rm d}s
=0.
\end{align}
Substituting the TBCs on $\Gamma_k, k=1, 2$ into \eqref{Identity}, we arrive at the variational problem:
find $u\in H_{{\rm qp}, \Gamma_c}^2(\Omega)$ such that 
\begin{equation}\label{variationalTBC}
	a(u, v)=\int_{\Gamma_1} \left(p_1\bar{v} + p_2 \partial_\nu \bar{v}\right){\rm d}s\quad \forall\, v\in H_{{\rm qp}, \Gamma_c}^2(\Omega),
\end{equation}
where the sesquilinear form $a(u, v):  H_{{\rm qp}, \Gamma_c}^2(\Omega)\times  
 H_{{\rm qp}, \Gamma_c}^2(\Omega)\rightarrow \mathbb{C}$ is defined as 
\begin{align}\label{tbcauv}
	a(u, v)=\int_{\Omega} \bigg[\mu\Delta u\Delta \bar{v}+(1-\mu)\sum\limits_{i,j=1}^2
	 \frac{\partial^2 u}{\partial x_i\partial x_j}
\frac{\partial^2 \bar{v}}{\partial x_i\partial x_j}-\kappa^4 u\bar{v}\bigg]{\rm d}x\notag\\
-\sum_{k=1}^2 \int_{\Gamma_k}  (\mathbb T^{(k)}\boldsymbol{u})\cdot\overline{\boldsymbol v}{\rm d}s
\end{align}
with $\boldsymbol{u}, \boldsymbol{v}$, and $\mathbb T^{(k)}$ given by 
\[
 \boldsymbol{u}=\begin{bmatrix}
                 u\\
                 \partial_\nu u
                \end{bmatrix},\quad \boldsymbol{v}=\begin{bmatrix}
                 v\\
                 \partial_\nu v
                \end{bmatrix},\quad 
\mathbb T^{(k)}=\begin{bmatrix}
                 T_{11}^{(k)} & T_{12}^{(k)}\\[5pt]
                 T_{21}^{(k)} & T_{22}^{(k)}
                \end{bmatrix}.
\]

The following trace theorem can be found in \cite[Theorem 1.1.6]{BS-FEM}.

\begin{lemma}\label{traceTheorem}
Let $\Omega$ be a Lipschitz domain. Then, there is a positive constant $C$ for which
\[
\|u\|_{L^2(\partial \Omega)}\leq C\|u\|_{L^2(\Omega)}^{1/2} \|u\|_{H^1(\Omega)}^{1/2}\quad
\forall\, u\in H^{1}(\Omega).
\]
\end{lemma}

\begin{theorem}\label{MainResult1}
The variational problem \eqref{variationalTBC} has a unique weak solution $u\in H^2_{{\rm qp}, \Gamma_c}(\Omega)$ except for a discrete set of wavenumbers $\kappa$.
\end{theorem}

\begin{proof}
It follows from Lemma \ref{BoundTBC} and the trace theorem (cf. \cite[Lemmas 2.2 and 2.3]{BL-IP-2014}) that the continuity of the sesquilinear form \eqref{tbcauv} is evident. It can be shown from \cite{BCF-CMS-2019} that there exist  positive constants $c_1$ and $c_2$ such that 
\begin{equation}\label{atbcp}
	\int_{\Omega} \bigg[\mu|\Delta u|^2
		+(1-\mu)\sum\limits_{i,j=1}^2\bigg|\frac{\partial^2 u}{\partial x_i\partial x_j}\bigg|^2
			-\kappa^4 |u|^2\bigg]{\rm d}x\geq c_1\|u\|_{H^2(\Omega)}^2-c_2\|u\|_{L^2(\Omega)}^2.
\end{equation}
By combining Lemmas \ref{BoundTBC}, \ref{PositiveTBC}, and \ref{traceTheorem} with the Cauchy inequality, we deduce 
\begin{align*}
	-\Re \sum_{k=1}^2 \int_{\Gamma_k} (\mathbb T^{(k)}\boldsymbol{u})\cdot
 	\overline{\boldsymbol u}{\rm d}s
 &\geq -c_3\big(\|u\|^2_{L^2(\Gamma_1\cup\Gamma_2)}+\|\partial_{x_2} u\|^2_{L^2(\Gamma_1\cup\Gamma_2)}\big)\\
 &\geq -c_4\big(\|u\|_{L^2(\Omega)}\|u\|_{H^1(\Omega)}+\|\nabla u\|_{L^2(\Omega)}\|\nabla u\|_{H^1(\Omega)}\big)\\
 &\geq -c_5\|u\|^2_{H^1(\Omega)}-c_6\|u\|^2_{H^2(\Omega)},
\end{align*}
where $c_6>0$ is sufficiently small. Combining the above inequality with \eqref{atbcp}, we verify that the sesquilinear form \eqref{tbcauv} satisfies the G\r{a}rding inequality
\[
	\Re a(u,u)\geq c_7 \|u\|_{H^2(\Omega)}-c_8\|u\|_{H^1(\Omega)},
\]
which completes the proof by applying the Fredholm alternative theorem.
\end{proof}

\section{The PML problem}\label{Section5}

This section focuses on the PML problem, aiming to establish its well-posedness while also providing a convergence analysis of the PML solution. 

\subsection{The PML formulation}

Denote by $\Omega_k^{\rm PML}$, for $k=1, 2$, the PML regions above and below the interfaces $\Gamma_k$, respectively. Assume that the thickness of each region is $\delta_k$ and denote the outer boundary of $\Omega_k^{\rm PML}$ by $\Gamma_k^{\rm PML}$. Let $\Omega^{\rm PML}=\Omega\cup\Omega_1^{\rm PML}\cup\Omega_2^{\rm PML}$ be the domain in which the PML problem is formulated. The schematic of the PML problem is depicted in Figure \ref{pml}.

\begin{figure}
\centering
\includegraphics[width=0.25\textwidth]{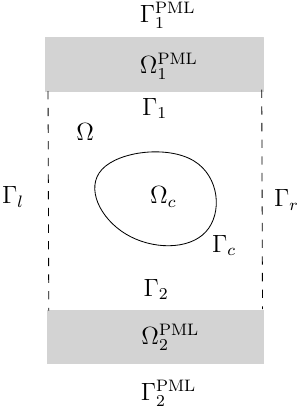}
\caption{The schematic of the PML problem.}
\label{pml}
\end{figure}

The PML parameters in $\Omega_k^{\rm PML}$ are introduced by the complex coordinate stretching (cf. \cite{TC-IEEE-1997}):
\[
\tilde{x}_2=\varphi(x_2)=\int_0^{x_2} s(t){\rm d}t,\quad s(t)=\left\{
\begin{aligned}
& 1 & &\text{if} ~ h_2\leq t\leq h_1,\\
& 1+{\rm i}\sigma_0\left(\frac{t-h_1}{\delta_1}\right)^m & &\text{if} ~  t>h_1,\\
& 1+{\rm i}\sigma_0\left(\frac{h_2-t}{\delta_2}\right)^m & &\text{if} ~ t<h_2,
\end{aligned}
\right.
\]
where $\sigma_0$ is a positive constant, and $m\geq 2$ is an integer.

Let $\tilde{u}$ be the solution of the total field for the PML problem in the complex coordinates $\tilde x=(x_1, \tilde x_2)$. It satisfies
\begin{equation}\label{MBHE}
\left\{
\begin{aligned}
&\tilde{\Delta}^2\tilde{u}-\kappa^4\tilde{u}=w & &{\rm in} ~ \Omega^{\rm PML},\\
&\tilde{u}=u^i,\quad\partial_{\nu}\tilde{u}=\partial_{\nu} u^i & &{\rm on} ~ \Gamma_1^{\rm PML},\\
&\tilde{u}=0,\quad\partial_{\nu}\tilde{u}=0 & &{\rm on} ~ \Gamma_2^{\rm PML},\\
\end{aligned}
\right.
\end{equation}
where $\tilde{\Delta}$ is the Laplacian operator in the complex coordinates and is given by 
\[
\tilde{\Delta}\tilde{u}=\frac{\partial^2 \tilde{u}}{\partial x_1^2}+\frac{\partial^2 \tilde{u}}{\partial \tilde{x}_2^2},
\]
and 
\[
w(\tilde x)=\left\{
\begin{aligned}
& \tilde{\Delta}^2 u^i-\kappa^4 u^i & &\text{in} ~ \Omega_1^{\rm PML},\\
& 0  & &\text{in} ~ \Omega\setminus\overline{\Omega_1^{\rm PML}}. 
\end{aligned}
\right.
\]

Here, the complex coordinate variable $\tilde x=(x_1, \tilde{x}_2)$ is considered to be in $\Omega^{\rm PML}_k$ or on $\Gamma_k^{\rm PML}$ if $x=(x_1, \psi(\tilde{x}_2))$ is in $\Omega^{\rm PML}_k$ or on $\Gamma_k^{\rm PML}$, where $\psi$ is the inverse function of $\varphi$, i.e., $x_2=\psi(\tilde x_2)=\varphi^{-1}(\tilde x_2)$.

\subsection{TBC for the PML problem}

To equivalently formulate the PML problem from the domain $\Omega^{\rm PML}$ to the domain $\Omega$, we investigate the TBCs of the PML problem \eqref{MBHE} on the interfaces $\Gamma_k$. 

First, we deduce the TBC on $\Gamma_1$. Consider the scattered field 
\[
	\tilde{u}^s(x_1, \varphi(x_2))=\tilde{u}(x_1, \varphi(x_2))-u^i(x_1, x_2). 
\]
It follows from \eqref{MBHE} that the scattered field $\tilde{u}^s$ satisfies
\begin{equation}\label{PMLOmega1}
\left\{
\begin{aligned}
&\tilde{\Delta}\tilde{u}^s-\kappa^4\tilde{u}^s=0 & &{\rm in} ~ \Omega_1^{\rm PML},\\
&\tilde{u}^s=\hat{f}_1,\quad\partial_{\nu}\tilde{u}^s=\hat{g}_1 & &{\rm on} ~ \Gamma_1,\\
&\tilde{u}^s=0,\quad\partial_{\nu}\tilde{u}^s=0 & &{\rm on} ~ \Gamma_1^{\rm PML},\\
\end{aligned}
\right.
\end{equation}
where $(\tilde{u}^s, \partial_{\nu} \tilde{u}^s)=(\hat{f}_1, \hat{g}_1)$ are the Dirichlet data for the scattered field $\tilde{u}^s$ on $\Gamma_1$. As $\hat f_1$ and $\hat g_1$ are quasi-periodic functions of $x_1$ with a period $\Lambda$, they have the Fourier series expansions
\[
\hat{f}_1(x_1)=\sum\limits_{n\in\mathbb{Z}} \hat{f}_1^{(n)}e^{{\rm i}\alpha_n x_1},\quad
\hat{g}_1(x_1)=\sum\limits_{n\in\mathbb{Z}} \hat{g}_1^{(n)}e^{{\rm i}\alpha_n x_1}.
\]

Since the scattered field $\tilde u^s$ satisfies $\tilde{\Delta}\tilde{u}^s-\kappa^4\tilde{u}^s=0$ in $\Omega_1^{\rm PML}$, we can verify that $\tilde u^s$ admits the following analytical expression in $\Omega_1^{\rm PML}$:
 \begin{align}
\tilde{u}^s(x_1, \tilde{x}_2)&=\sum\limits_{n\in\mathbb{Z}}\Big[
	W_1^{(n)}e^{-{\rm i}\beta_n(\tilde{x}_2-h_1)}+V_1^{(n)}e^{{\rm i}\beta_n(\tilde{x}_2-h_1)}\notag\\
	&\quad +X_1^{(n)}e^{\gamma_n(\tilde{x}_2-h_1)}	+Y_1^{(n)}e^{-\gamma_n(\tilde{x}_2-h_1)}\Big]e^{{\rm i}\alpha_n x_1}\label{uMBHE1}.
\end{align}
Substituting \eqref{uMBHE1} into the boundary conditions in \eqref{PMLOmega1}, we obtain a linear system of algebraic equations for the Fourier coefficients $W_1^{(n)}, V_1^{(n)}, X_1^{(n)}$, and $Y_1^{(n)}$:
\begin{equation}\label{System1}
\left\{
\begin{aligned}
& W_1^{(n)}+V_1^{(n)}+X_1^{(n)}+Y_1^{(n)}=\hat{f}_1^{(n)},\\	
& -{\rm i}\beta_n W_1^{(n)}+{\rm i}\beta_n V_1^{(n)}+\gamma_n X_1^{(n)}-\gamma_n Y_1^{(n)}=\hat{g}_1^{(n)},\\
& W_1^{(n)}e^{-{\rm i}\beta_n \tilde{h}_1}+V_1^{(n)}e^{{\rm i}\beta_n \tilde{h}_1}
+X_1^{(n)}e^{\gamma_n \tilde{h}_1}+Y_1^{(n)}e^{-\gamma_n \tilde{h}_1}=0,\\
& -{\rm i}\beta_n W_1^{(n)}e^{-{\rm i}\beta_n \tilde{h}_1}+{\rm i}\beta_n V_1^{(n)}e^{{\rm i}\beta_n \tilde{h}_1}+\gamma_n X_1^{(n)}e^{\gamma_n \tilde{h}_1}-\gamma_n Y_1^{(n)}e^{-\gamma_n \tilde{h}_1}=0,
\end{aligned}
\right.
\end{equation}
where
\begin{equation}\label{h1}
	\tilde{h}_1=\varphi(h_1+\delta_1)-h_1=\int_0^{h_1+\delta_1} s(t){\rm d}t-h_1
	=\delta_1\Big(1+\frac{{\rm i}\sigma_0}{m+1}\Big).
\end{equation}

Through tedious yet straightforward calculations, we solve the linear system \eqref{System1} and obtain the solutions
for the Fourier coefficients:
\begin{align*}
W_1^{(n)}&=\frac{1}{D^{(n)}_1}\bigg\{\hat{f}_1^{(n)}\Big[\left(\gamma_n\beta_n
	-{\rm i}\gamma_n^2\right)e^{2{\rm i}\beta_n\tilde{h}_1}
	-2\beta_n\gamma_n e^{({\rm i}\beta_n+\gamma_n) \tilde{h}_1}\\
	&\quad +\left(\gamma_n\beta_n+{\rm i}\gamma_n^2\right)e^{2({\rm i}\beta_n+\gamma_n) \tilde{h}_1}
	 \Big]
	 -\hat{g}_1^{(n)}\Big[\left(\beta_n-{\rm i}\gamma_n\right)e^{2{\rm i}\beta_n\tilde{h}_1}\\
&\quad	 +2{\rm i}\gamma_n e^{({\rm i}\beta_n+\gamma_n) \tilde{h}_1}
	 -\left({\rm i}\gamma_n+\beta_n\right)e^{2({\rm i}\beta_n+\gamma_n) \tilde{h}_1}\Big]\bigg\},\\
V_1^{(n)}&=\frac{1}{D^{(n)}_1}\bigg\{\hat{f}_1^{(n)}\left[\left(\beta_n\gamma_n+{\rm i}\gamma_n^2\right)
	+\left(\gamma_n\beta_n-{\rm i}\gamma_n^2\right)e^{2\gamma_n\tilde{h}_1}
	-2\beta_n\gamma_n e^{({\rm i}\beta_n+\gamma_n) \tilde{h}_1}
	 \right]\\
	 &\quad-\hat{g}_1^{(n)}\left[\left(\beta_n+{\rm i}\gamma_n\right)-\left(\beta_n-{\rm i}\gamma_n\right)e^{2\gamma_n\tilde{h}_1}
	 -2{\rm i}\gamma_n e^{({\rm i}\beta_n+\gamma_n) \tilde{h}_1}
	\right]\bigg\},\\
X_1^{(n)}&=\frac{1}{D^{(n)}_1}\bigg\{\hat{f}_1^{(n)}\left[(\beta_n\gamma_n-{\rm i}\beta_n^2)
	+\left({\rm i}\beta_n^2+\beta_n\gamma_n\right)e^{2{\rm i}\beta_n\tilde{h}_1}
	-2\beta_n\gamma_n e^{({\rm i}\beta_n+\gamma_n) \tilde{h}_1}
	 \right]\\
	 &\quad-\hat{g}_1^{(n)}\left[-\left(\beta_n+{\rm i}\gamma_n\right)
	 	-\left(\beta_n-{\rm i}\gamma_n\right)e^{2{\rm i}\beta_n\tilde{h}_1}
	 +2\beta_n e^{({\rm i}\beta_n+\gamma_n) \tilde{h}_1}
	\right]\bigg\},\\
Y_1^{(n)}&=\frac{1}{D^{(n)}_1}\bigg\{\hat{f}_1^{(n)}\Big[
		\left({\rm i}\beta_n^2+\beta_n\gamma_n\right)e^{2\gamma_n\tilde{h}_1}
		-2\beta_n\gamma_n e^{({\rm i}\beta_n+\gamma_n) \tilde{h}_1}\\
&\quad		+\left(\beta_n\gamma_n-{\rm i}\beta_n^2\right)e^{2({\rm i}\beta_n+\gamma_n) \tilde{h}_1}
	 \Big]-\hat{g}_1^{(n)}\Big[\left(\beta_n-{\rm i}\gamma_n\right)e^{2\gamma_n\tilde{h}_1}\\
&\quad	 -2\beta_n e^{({\rm i}\beta_n+\gamma_n) \tilde{h}_1}
	 +\left({\rm i}\gamma_n+\beta_n\right)e^{2({\rm i}\beta_n+\gamma_n) \tilde{h}_1}\Big]\bigg\}, 
\end{align*}
where the denominator $D^{(n)}_1$ is defined as 
\begin{align}\label{Denomiantor1}
D^{(n)}_1&=-8\beta_n\gamma_n e^{({\rm i}\beta_n+\gamma_n)\tilde{h}_1}+\left(-{\rm i}\beta_n^2+{\rm i}\gamma_n^2+2\beta_n\gamma_n\right)\big(1+e^{2({\rm i}\beta_n+\gamma_n)\tilde{h}_1}\big)\notag\\ 
&\quad +\left({\rm i}\beta_n^2-{\rm i}\gamma_n^2+2\beta_n\gamma_n\right)\big(e^{2{\rm i}\beta_n \tilde{h}_1}+e^{2\gamma_n\tilde{h}_1}\big).
\end{align}

It is clear to note from \eqref{h1} that both of the real and imaginary parts of $\tilde{h}_1$ are positive.  
If the thickness of the layer $\delta_1$ is sufficiently large, i.e., $\Re\tilde h_1=\delta_1$ is sufficiently large, then the leading term in \eqref{Denomiantor1} is the one containing $e^{2\gamma_n \tilde{h}_1}$. A simple calculations yields  
\[
	\left|{\rm i}\beta_n^2-{\rm i}\gamma_n^2+2\beta_n\gamma_n\right|^2
	=|\gamma_n+{\rm i}\beta_n|^4\neq 0\quad \forall \, n\in\mathbb Z,
\]
which ensures that $D^{(n)}_1$ is non-zero for sufficiently large $\delta_1$.

Substituting \eqref{uMBHE1} into \eqref{MNupper}, we obtain the TBC of the scattered field $\tilde u^s$ for the PML problem on $\Gamma_1$:
\begin{equation}\label{PMLGamma1us}
N_1\tilde{u}^s=\hat{T}_{11}^{(1)} \hat{f}_1+\hat{T}_{12}^{(1)}\hat{g}_1,\quad 
M_1\tilde{u}^s=\hat{T}_{21}^{(1)}\hat{f}_1+\hat{T}_{22}^{(1)}\hat{g}_1,
\end{equation}
where the DtN operators $\hat{T}_{ij}^{(1)}, i,j=1,2$ are given by 
\begin{align*}
(\hat{T}^{(1)}_{11} \hat{f}_1)(x_1)&= -\sum\limits_{n\in\mathbb{Z}}\frac{e^{{\rm i}\alpha_n x_1}}{D^{(n)}_1}\hat{f}_1^{(n)}\Big\{-\left({\rm i}\beta_n^4\gamma_n-\beta_n^3\gamma_n^2+{\rm i}\beta_n^2\gamma_n^3-\beta_n\gamma_n^4\right)\\
&\quad +\left({\rm i}\beta_n^4\gamma_n+\beta_n^3\gamma_n^2+{\rm i}\beta_n^2\gamma_n^3+\beta_n\gamma_n^4\right)	 
e^{2{\rm i}\beta_n\tilde{h}_1} \\
& \quad -\left({\rm i}\beta_n^4\gamma_n+\beta_n^3\gamma_n^2+{\rm i}\beta_n^2\gamma_n^3 \beta_n\gamma_n^4\right)e^{2\gamma_n\tilde{h}_1}\\
&\quad +\left({\rm i}\beta_n^4\gamma_n-\beta_n^3\gamma_n^2+{\rm i}\beta_n^2\gamma_n^3-\beta_n\gamma_n^4\right)
e^{2({\rm i}\beta_n+\gamma_n) \tilde{h}_1} \Big\},
\end{align*}
\begin{align*}
(\hat{T}^{(1)}_{12} \hat{g}_1)(x_1) &=-\sum\limits_{n\in\mathbb{Z}}\frac{e^{{\rm i}\alpha_n x_1}}{D^{(n)}_1}\hat{g}_1^{(n)}\bigg\{\Big(\frac{\rm i}{2}\mu\beta_n^4+(1-\mu)\beta_n^3\gamma_n+(2-\mu){\rm i}\beta_n^2\gamma_n^2\\
&\quad -(1-\mu)\beta_n\gamma_n^3+\frac{\mu}{2}{\rm i}\gamma_n^4\Big)
 +\Big(-\frac{\rm i}{2}\mu\beta_n^4+(1-\mu)\beta_n^3\gamma_n-(2-\mu){\rm i}\beta_n^2\gamma_n^2\\
&\quad -(1-\mu)\beta_n\gamma_n^3 -\frac{\rm i}{2}\mu\gamma_n^4\Big)e^{2{\rm i}\beta_n\tilde{h}_1}
  +\Big(-\frac{\rm i}{2}\mu\beta_n^4+(1-\mu)\beta_n^3\gamma_n\\
 &\quad -(2-\mu){\rm i}\beta_n^2\gamma_n^2
 -(1-\mu)\beta_n\gamma_n^3 -\frac{\rm i}{2}\mu\gamma_n^4\Big)e^{2\gamma_n\tilde{h}_1}\\
&\quad +(4-4\mu)\left(\beta_n\gamma_n^3-\beta_n^3\gamma_n\right)e^{({\rm i}\beta_n+\gamma_n) \tilde{h}_1}
+\Big(\frac{\rm i}{2}\mu\beta_n^4+(1-\mu)\beta_n^3\gamma_n\\
&\quad +(2-\mu){\rm i}\beta_n^2\gamma_n^2 -(1-\mu)\beta_n\gamma_n^3+\frac{\mu}{2}{\rm i}\gamma_n^4\Big)e^{2({\rm i}\beta_n+\gamma_n) \tilde{h}_1}\bigg\}, 
\end{align*}
\begin{align*}
(\hat{T}_{21}^{(1)}\hat{f}_1)(x_1)&=-\sum\limits_{n\in\mathbb{Z}}\frac{e^{{\rm i}\alpha_n x_1}}{D^{(n)}_1}\hat{f}_1^{(n)}\bigg\{\Big(\frac{\rm i}{2}\mu\beta_n^4+(1-\mu)\beta_n^3\gamma_n+(2-\mu){\rm i}\beta_n^2\gamma_n^2\\
&\quad -(1-\mu)\beta_n\gamma_n^3+\frac{\rm i}{2}\mu\gamma_n^4\Big)
 +\Big(-\frac{\rm i}{2}\mu\beta_n^4+(1-\mu)\beta_n^3\gamma_n-(2-\mu){\rm i}\beta_n^2\gamma_n^2\\
&\quad -(1-\mu)\beta_n\gamma_n^3-\frac{\rm i}{2}\mu\gamma_n^4\Big)e^{2{\rm i}\beta_n\tilde{h}_1}
 +\Big(-\frac{\rm i}{2}\mu\beta_n^4+(1-\mu)\beta_n^3\gamma_n\\
 &\quad -(2-\mu){\rm i}\beta_n^2\gamma_n^2
-(1-\mu)\beta_n\gamma_n^3-\frac{\rm i}{2}\mu\gamma_n^4\Big)e^{2\gamma_n\tilde{h}_1}\\
&\quad +(4-4\mu)\left(\gamma_n^3\beta_n-\beta_n^3\gamma_n\right)e^{({\rm i}\beta_n+\gamma_n) \tilde{h}_1}
 +\Big(\frac{\rm i}{2}\mu\beta_n^4+(1-\mu)\beta_n^3\gamma_n\\
&\quad +(2-\mu){\rm i}\beta_n^2\gamma_n^2
-(1-\mu)\beta_n\gamma_n^3+\frac{\rm i}{2}\mu\gamma_n^4\Big)e^{2({\rm i}\beta_n+\gamma_n) \tilde{h}_1}\bigg\},
\end{align*}
and
\begin{align*}
(\hat{T}_{22}^{(1)}\hat{g}_1)(x_1)&=-\sum\limits_{n\in\mathbb{Z}}\frac{e^{{\rm i}\alpha_n x_1}}{D^{(n)}_1}\hat{g}_1^{(n)}\Big\{ -\left(\beta_n^3+{\rm i}\beta_n^2\gamma_n+\beta_n\gamma_n^2+{\rm i}\gamma_n^3\right)	\\
&\quad -\left(\beta_n^3-{\rm i}\beta_n^2\gamma_n+\beta_n\gamma_n^2-{\rm i}\gamma_n^3\right)e^{2{\rm i}\beta_n\tilde{h}_1}\\
&\quad +\left(\beta_n^3-{\rm i}\beta_n^2\gamma_n+\beta_n\gamma_n^2-{\rm i}\gamma_n^3\right)e^{2\gamma_n\tilde{h}_1}\\
&\quad +\left(\beta_n^3+{\rm i}\beta_n^2\gamma_n+\beta_n\gamma_n^2+{\rm i}\gamma_n^3\right)	
e^{2({\rm i}\beta_n+\gamma_n) \tilde{h}_1}\Big\}.
\end{align*}

Let $(\tilde u, \partial_\nu\tilde u)=(f_1, g_1)$ be the Dirichlet data of the total field $\tilde u$ on $\Gamma_1$. Utilizing \eqref{PMLGamma1us} and noting $f_1=\hat f_1+u^i$ and $g_1=\hat g_1+\partial_{x_2}u^i$, the TBC for the total field $\tilde{u}$ on $\Gamma_1$ can be formulated as 
 \begin{equation}\label{DtNGamma1PML}
N_1\tilde{u}=\hat{T}_{11}^{(1)} f_1+\hat{T}_{12}^{(1)}g_1+\hat{p}_1,\quad 
M_1\tilde{u}=\hat{T}_{21}^{(1)}f_1+\hat{T}_{22}^{(1)}g_1+\hat{p}_2, 
\end{equation}
where
\begin{equation}\label{hp12}
\hat{p}_1(x_1) =N_1 u^i-\hat{T}_{11}^{(1)} u^i-\hat{T}_{12}^{(1)}\partial_{x_2} u^i,\quad 
\hat{p}_2(x_1) =M_1 u^i-\hat{T}_{21}^{(1)} u^i-\hat{T}_{22}^{(1)}\partial_{x_2} u^i.
\end{equation}

The TBC can be similarly deduced on $\Gamma_2$. Derived from the biharmonic wave equation in \eqref{MBHE}, the PML solution $\tilde{u}$ exhibits the following analytical expansion in $\Omega_2^{\rm PML}$:
\begin{align}\label{uMBHE2}
\tilde{u}(x_1, \tilde{x}_2)&=\sum\limits_{n\in\mathbb{Z}}\Big[
	W_2^{(n)}e^{-{\rm i}\beta_n(\tilde{x}_2-h_2)}+V_2^{(n)}e^{{\rm i}\beta_n(\tilde{x}_2-h_2)}\notag\\
&\quad 	+X_2^{(n)}e^{\gamma_n(\tilde{x}_2-h_2)} +Y_2^{(n)}e^{-\gamma_n(\tilde{x}_2-h_2)}\Big] e^{{\rm i}\alpha_n x_1}.
\end{align}
Substituting \eqref{uMBHE2} into the boundary conditions in \eqref{MBHE}, we obtain a linear system for the Fourier coefficients $W_2^{(n)}, V_2^{(n)}, X_2^{(n)}$, and $Y_2^{(n)}$:
\begin{equation}\label{PML2}
\left\{
\begin{aligned}
& W_2^{(n)}+V_2^{(n)}+X_2^{(n)}+Y_2^{(n)}=f_2^{(n)},\\
& -{\rm i}\beta_n W_2^{(n)}+{\rm i}\beta_n V_2^{(n)} +\gamma_n X_2^{(n)}-\gamma_n Y_2^{(n)}=-g_2^{(n)},\\
& W_2^{(n)}e^{-{\rm i}\beta_n \tilde{h}_2}+V_2^{(n)}e^{{\rm i}\beta_n \tilde{h}_2} +X_2^{(n)}e^{\gamma_n \tilde{h}_2}+ Y_2^{(n)}e^{-\gamma_n \tilde{h}_2}=0,\\
& -{\rm i}\beta_n W_2^{(n)}e^{-{\rm i}\beta_n \tilde{h}_2}+{\rm i}\beta_n V_2^{(n)}e^{{\rm i}\beta_n \tilde{h}_2} +\gamma_n X_2^{(n)}e^{\gamma_n \tilde{h}_2}-\gamma_n Y_2^{(n)}e^{-\gamma_n \tilde{h}_2}=0,
\end{aligned}
\right.
\end{equation}
where $(\tilde{u}, \partial_{\nu} \tilde{u})=(f_2, g_2)$ represent the Dirichlet data of the total field on $\Gamma_2$ and have the Fourier series expansions
\[
	f_2=\sum\limits_{n\in\mathbb{Z}} f_2^{(n)}e^{{\rm i}\alpha_n x_1},\quad
	g_2=\sum\limits_{n\in\mathbb{Z}} g_2^{(n)}e^{{\rm i}\alpha_n x_1}, 
\]
and
\begin{equation}\label{h2}
	\tilde{h}_2:=\varphi(h_2-\delta_2)-h_2=\int_0^{h_2-\delta_2} s(t){\rm d}t-h_2
	=-\delta_2\left(1+\frac{{\rm i}\sigma_0}{m+1}\right).
\end{equation}

Upon solving the linear system \eqref{PML2}, we have 
\begin{align*}
W_2^{(n)}&=\frac{1}{D^{(n)}_2}\bigg\{f_2^{(n)}\Big[\left(\gamma_n\beta_n-{\rm i}\gamma_n^2\right)
        e^{2{\rm i}\beta_n\tilde{h}_2} -2\beta_n\gamma_n e^{({\rm i}\beta_n+\gamma_n) \tilde{h}_2}\\
&\quad +\left(\gamma_n\beta_n+{\rm i}\gamma_n^2\right)e^{2({\rm i}\beta_n+\gamma_n) \tilde{h}_2}
	 \Big]
	 +g_2^{(n)}\Big[\left(\beta_n-{\rm i}\gamma_n\right)e^{2{\rm i}\beta_n\tilde{h}_2}\\
&\quad +2{\rm i}\gamma_ne^{({\rm i}\beta_n+\gamma_n) \tilde{h}_2}
	 -\left({\rm i}\gamma_n+\beta_n\right)e^{2({\rm i}\beta_n+\gamma_n) \tilde{h}_2}\Big]\bigg\},\\
\end{align*}
\begin{align*}
	 V_2^{(n)}&=\frac{1}{D^{(n)}_2}\bigg\{f_2^{(n)}\left[\left(\beta_n\gamma_n+{\rm i}\gamma_n^2\right)
	+\left(\gamma_n\beta_n-{\rm i}\gamma_n^2\right)e^{2\gamma_n\tilde{h}_2}
	-2\beta_n\gamma_n e^{({\rm i}\beta_n+\gamma_n) \tilde{h}_2}\right]\\
	 &\quad+g_2^{(n)}\left[\left(\beta_n+{\rm i}\gamma_n\right)-\left(\beta_n-{\rm i}\gamma_n\right)e^{2\gamma_n\tilde{h}_2}-2{\rm i}\gamma_n e^{({\rm i}\beta_n+\gamma_n) \tilde{h}_2}
	\right]\bigg\},
\end{align*}
\begin{align*}
X_2^{(n)}&=\frac{1}{D^{(n)}_2}\bigg\{f_2^{(n)}\left[\left(\beta_n\gamma_n-{\rm i}\beta_n^2\right)
	+\left({\rm i}\beta_n^2+\beta_n\gamma_n\right)e^{2{\rm i}\beta_n\tilde{h}_2}
	-2\beta_n\gamma_n e^{({\rm i}\beta_n+\gamma_n) \tilde{h}_2}\right]\\
	 &\quad+g_2^{(n)}\left[-\left(\beta_n+{\rm i}\gamma_n\right)
	 	-\left(\beta_n-{\rm i}\gamma_n\right)e^{2{\rm i}\beta_n\tilde{h}_2}
	 +2\beta_n e^{({\rm i}\beta_n+\gamma_n) \tilde{h}_2}\right]\bigg\},
\end{align*}
\begin{align*}
Y_2^{(n)}&=\frac{1}{D^{(n)}_2}\bigg\{f_2^{(n)}\Big[
		\left({\rm i}\beta_n^2+\beta_n\gamma_n\right)e^{2\gamma_n\tilde{h}_2}
		-2\beta_n\gamma_n e^{({\rm i}\beta_n+\gamma_n) \tilde{h}_2}\\
&\quad +\left(\beta_n\gamma_n-{\rm i}\beta_n^2\right)e^{2({\rm i}\beta_n+\gamma_n) \tilde{h}_2}
	 \Big]+g_2^{(n)}\Big[\left(\beta_n-{\rm i}\gamma_n\right)e^{2\gamma_n\tilde{h}_2}\\
&\quad -2\beta_ne^{({\rm i}\beta_n+\gamma_n) \tilde{h}_2}
	 +\left({\rm i}\gamma_n+\beta_n\right)e^{2({\rm i}\beta_n+\gamma_n) \tilde{h}_2}\Big]\bigg\},
\end{align*}
where the denominator $D^{(n)}_2$ is given by 
\begin{eqnarray}\label{Denomiantor2}
D^{(n)}_2 &=-8\beta_n\gamma_n e^{({\rm i}\beta_n+\gamma_n)\tilde{h}_2} + \left(-{\rm i}\beta_n^2+{\rm i}\gamma_n^2+2\beta_n\gamma_n\right)\big(1+e^{2({\rm i}\beta_n+\gamma_n)\tilde{h}_2}\big)\notag\\
& \quad +\left({\rm i}\beta_n^2-{\rm i}\gamma_n^2+2\beta_n\gamma_n\right)\big(e^{2{\rm i}\beta_n \tilde{h}_2}+e^{2\gamma_n\tilde{h}_2}\big).
\end{eqnarray}

By \eqref{h2}, both of the real and imaginary parts of $\tilde{h}_2$ are negative. If the thickness of the layer $
\delta_2$ is sufficiently large, i.e., $-\Re\tilde h_2=\delta_2$ is sufficiently large, then the leading term in \eqref{Denomiantor2} is the one containing $e^{2{\rm i}\beta_n \tilde{h}_2}$. Observing 
$|{\rm i}\beta_n^2-{\rm i}\gamma_n^2+2\beta_n\gamma_n|^2 =|\gamma_n+{\rm i}\beta_n|^4\neq 0$ for any $n\in\mathbb Z$, we 
deduce that the denominator $D^{(2)}_n$ in non-zero for sufficiently large $\delta_2$.

Substituting \eqref{uMBHE2} into \eqref{MNlower}, we obtain the TBC for the PML problem on $\Gamma_2$:
\begin{equation}\label{PMLGamma2}
N_2\tilde{u}=\hat{T}_{11}^{(2)} f_2+\hat{T}_{12}^{(2)}g_2,\quad 
M_2\tilde{u}=\hat{T}_{21}^{(2)}f_2+\hat{T}_{22}^{(2)}g_2,
\end{equation}
where the DtN operators $\hat{T}_{ij}^{(2)}, i,j=1, 2$ are defined as 
\begin{align*}
(\hat{T}^{(2)}_{11} f_2)(x_1)&= -\sum\limits_{n\in\mathbb{Z}}\frac{e^{{\rm i}\alpha_n x_1}}{D^{(2)}_n}f_2^{(n)}
\Big\{\left({\rm i}\beta_n^4\gamma_n-\beta_n^3\gamma_n^2+{\rm i}\beta_n^2\gamma_n^3-\beta_n\gamma_n^4\right)\\
&\quad-\left({\rm i}\beta_n^4\gamma_n+\beta_n^3\gamma_n^2+{\rm i}\beta_n^2\gamma_n^3+\beta_n\gamma_n^4\right)	 
e^{2{\rm i}\beta_n\tilde{h}_2}\\
&\quad +\left({\rm i}\beta_n^4\gamma_n+\beta_n^3\gamma_n^2+{\rm i}\beta_n^2\gamma_n^3+\beta_n\gamma_n^4\right)	 
e^{2\gamma_n\tilde{h}_2}\\
&\quad -\left({\rm i}\beta_n^4\gamma_n-\beta_n^3\gamma_n^2+{\rm i}\beta_n^2\gamma_n^3-\beta_n\gamma_n^4\right)
e^{2({\rm i}\beta_n+\gamma_n) \tilde{h}_2}\Big\},
\end{align*}
\begin{align*}
(\hat{T}^{(2)}_{12} g_2)(x_1) &=-\sum\limits_{n\in\mathbb{Z}}\frac{e^{{\rm i}\alpha_n x_1}}{D^{(2)}_n}g_2^{(n)}
\bigg\{\Big(\frac{\rm i}{2}\mu\beta_n^4+(1-\mu)\beta_n^3\gamma_n+(2-\mu){\rm i}\beta_n^2\gamma_n^2\\
&\quad -(1-\mu)\beta_n\gamma_n^3+\frac{\mu}{2}{\rm i}\gamma_n^4\Big)
 +\Big(-\frac{\rm i}{2}\mu\beta_n^4+(1-\mu)\beta_n^3\gamma_n-(2-\mu){\rm i}\beta_n^2\gamma_n^2\\
&\quad -(1-\mu)\beta_n\gamma_n^3 -\frac{\rm i}{2}\mu\gamma_n^4\Big)e^{2{\rm i}\beta_n\tilde{h}_2} +\Big(-\frac{\rm i}{2}\mu\beta_n^4+(1-\mu)\beta_n^3\gamma_n \\
&\quad -(2-\mu){\rm i}\beta_n^2\gamma_n^2 -(1-\mu)\beta_n\gamma_n^3 -\frac{\rm i}{2}\mu\gamma_n^4\Big)e^{2\gamma_n\tilde{h}_2}\\
&\quad +(4-4\mu)\left(\beta_n\gamma_n^3-\beta_n^3\gamma_n\right)e^{({\rm i}\beta_n+\gamma_n) \tilde{h}_2}
 +\Big(\frac{\rm i}{2}\mu\beta_n^4+(1-\mu)\beta_n^3\gamma_n\\
 &\quad +(2-\mu){\rm i}\beta_n^2\gamma_n^2
-(1-\mu)\beta_n\gamma_n^3+\frac{\mu}{2}{\rm i}\gamma_n^4\Big)e^{2({\rm i}\beta_n+\gamma_n) \tilde{h}_2}\bigg\},
\end{align*}
\begin{align*}
(\hat{T}_{21}^{(2)}(n)f_2)(x_1)&=-\sum\limits_{n\in\mathbb{Z}}\frac{e^{{\rm i}\alpha_n x_1}}{D^{(2)}_n}f_2^{(n)}
\bigg\{\Big(\frac{\rm i}{2}\mu\beta_n^4+(1-\mu)\beta_n^3\gamma_n+(2-\mu){\rm i}\beta_n^2\gamma_n^2\\
&\quad -(1-\mu)\beta_n\gamma_n^3+\frac{\rm i}{2}\mu\gamma_n^4\Big) +\Big(-\frac{\rm i}{2}\mu\beta_n^4+(1-\mu)\beta_n^3\gamma_n-(2-\mu){\rm i}\beta_n^2\gamma_n^2\\
&\quad -(1-\mu)\beta_n\gamma_n^3-\frac{\rm i}{2}\mu\gamma_n^4\Big)e^{2{\rm i}\beta_n\tilde{h}_2}
+\Big(-\frac{\rm i}{2}\mu\beta_n^4+(1-\mu)\beta_n^3\gamma_n\\
&\quad -(2-\mu){\rm i}\beta_n^2\gamma_n^2 -(1-\mu)\beta_n\gamma_n^3-\frac{\rm i}{2}\mu\gamma_n^4\Big)e^{2\gamma_n\tilde{h}_2}\\
&\quad+(4-4\mu)\left(\gamma_n^3\beta_n-\beta_n^3\gamma_n\right)e^{({\rm i}\beta_n+\gamma_n) \tilde{h}_2}
+\Big(\frac{\rm i}{2}\mu\beta_n^4+(1-\mu)\beta_n^3\gamma_n\\
&\quad +(2-\mu){\rm i}\beta_n^2\gamma_n^2 -(1-\mu)\beta_n\gamma_n^3+\frac{\rm i}{2}\mu\gamma_n^4\Big)e^{2({\rm i}\beta_n+\gamma_n)\tilde{h}_2}\bigg\},
\end{align*}
and
\begin{align*}
(\hat{T}_{22}^{(2)}(n)g_2)(x_1)&=-\sum\limits_{n\in\mathbb{Z}}\frac{e^{{\rm i}\alpha_n x_1}}{D^{(2)}_n}g_2^{(n)}
\Big\{\left(\beta_n^3+{\rm i}\beta_n^2\gamma_n+\beta_n\gamma_n^2+{\rm i}\gamma_n^3\right)\\
&\quad +\left(\beta_n^3-{\rm i}\beta_n^2\gamma_n+\beta_n\gamma_n^2-{\rm i}\gamma_n^3\right)e^{2{\rm i}\beta_n\tilde{h}_2}\\
&\quad-\left(\beta_n^3-{\rm i}\beta_n^2\gamma_n+\beta_n\gamma_n^2-{\rm i}\gamma_n^3\right)e^{2\gamma_n\tilde{h}_2}\\
&\quad -\left(\beta_n^3+{\rm i}\beta_n^2\gamma_n+\beta_n\gamma_n^2+{\rm i}\gamma_n^3\right)e^{2({\rm i}\beta_n+\gamma_n) \tilde{h}_2}\Big\}.
\end{align*}

\begin{lemma}\label{BoundPML}
Assuming that the thickness of the PML layers $\delta_k$ for $k=1, 2$ are sufficiently large, the DtN operators $\hat{T}^{(k)}_{11}: H^{3/2}(\Gamma_k)\to H^{-3/2}(\Gamma_k)$, $\hat{T}^{(k)}_{12}: H^{1/2}(\Gamma_k)\to H^{-3/2}(\Gamma_k)$,
$\hat{T}^{(k)}_{21}: H^{3/2}(\Gamma_k)\to H^{-1/2}(\Gamma_k)$, and $\hat{T}^{(k)}_{22}: H^{1/2}(\Gamma_k)\to H^{-1/2}(\Gamma_k)$ are bounded.
\end{lemma}

\begin{proof}
It is sufficient to present the results for the DtN operators on $\Gamma_1$, as the corresponding outcomes can be similarly established for the DtN operators on $\Gamma_2$.

As demonstrated in the derivation of the TBC on $\Gamma_1$, when the thickness of the PML region is sufficiently large, the term containing $e^{2\gamma_n\tilde h_1}$ is dominant in both the denominator and the numerator. Noting \eqref{abg}, i.e., $\beta_n \sim\alpha_n$ and $\gamma_n\sim\alpha_n$ as $n\to\infty$, we assert that for any function $f\in H^{3/2}(\Gamma_1)$
\begin{align*}
	\|\hat{T}_{11}^{(1)} f\|^2_{H^{-3/2}(\Gamma_1)}&\lesssim \sum\limits_{n\in\mathbb{Z}}
		(1+\alpha_n^2)^{-3/2}\left|\frac{f^{(n)}}{D_1^{(n)}}
		\left({\rm i}\beta_n^4\gamma_n+\beta_n^3\gamma_n^2+{\rm i}\beta_n^2\gamma_n^3
		+\beta_n\gamma_n^4\right)e^{2\gamma_n\tilde{h}_1}\right| ^2\\
	&\lesssim \sum\limits_{n\in\mathbb{Z}}
		(1+\alpha_n^2)^{-3/2}\left|\frac{\left({\rm i}\beta_n^4\gamma_n
		+\beta_n^3\gamma_n^2+{\rm i}\beta_n^2\gamma_n^3
		+\beta_n\gamma_n^4\right) e^{2\gamma_n\tilde{h}_1}}{\left({\rm i}\beta_n^2
		-{\rm i}\gamma_n^2+2\beta_n\gamma_n\right)e^{2\gamma_n\tilde{h}_1}}
		\right| ^2 |f^{(n)}|^2\\
	&\lesssim \sum\limits_{n\in\mathbb{Z}}
		(1+\alpha_n^2)^{3/2}|f^{(n)}|^2=\|f\|^2_{H^{3/2}(\Gamma_1)}, 
\end{align*}
and 
\begin{align*}
\|\hat{T}_{21}^{(1)} f\|^2_{H^{-1/2}(\Gamma_1)} &\lesssim\sum\limits_{n\in\mathbb{Z}}(1+\alpha_n^2)^{-1/2}\\
&\hspace{-2cm}\times\left|\frac{f^{(n)}}{D_1^{(n)}}\Big(-\frac{\rm i}{2}\mu\beta_n^4+(1-\mu)\beta_n^3\gamma_n-(2-\mu){\rm i}\beta_n^2\gamma_n^2-(1-\mu)\beta_n\gamma_n^3-\frac{\rm i}{2}\mu\gamma_n^4\Big)e^{2\gamma_n\tilde{h}_1}\right|^2\\
&\lesssim\sum\limits_{n\in\mathbb{Z}}(1+\alpha_n^2)^{-1/2}\\
&\hspace{-2cm}\times\left|
	\frac{-\frac{\rm i}{2}\mu\beta_n^4+(1-\mu)\beta_n^3\gamma_n-(2-\mu){\rm i}\beta_n^2\gamma_n^2
	-(1-\mu)\beta_n\gamma_n^3-\frac{\rm i}{2}\mu\gamma_n^4}{{\rm i}\beta_n^2
	-{\rm i}\gamma_n^2+2\beta_n\gamma_n}\right|^2 |f^{(n)}|^2\\
&\lesssim\sum\limits_{n\in\mathbb{Z}}
		(1+\alpha_n^2)^{3/2} |f^{(n)}|^2=\|f\|^2_{H^{3/2}(\Gamma_1)}.
\end{align*}
Similarly, we have that for any function $f\in H^{1/2}(\Gamma_1)$
\begin{align*}
\|\hat{T}_{12}^{(1)} f\|^2_{H^{-3/2}(\Gamma_1)}&\lesssim\sum\limits_{n\in\mathbb{Z}}(1+\alpha_n^2)^{-3/2}\\
&\hspace{-2cm}\times\left|\frac{f^{(n)}}{D_1^{(n)}}\Big(-\frac{\rm i}{2}\mu\beta_n^4+(1-\mu)\beta_n^3\gamma_n-(2-\mu){\rm i}\beta_n^2\gamma_n^2-(1-\mu)\beta_n\gamma_n^3-\frac{\rm i}{2}\mu\gamma_n^4\Big)e^{2\gamma_n\tilde{h}_1}\right|^2\\
&\lesssim\sum\limits_{n\in\mathbb{Z}}(1+\alpha_n^2)^{-3/2}\\
&\hspace{-2cm}\times\left|\frac{-\frac{\rm i}{2}\mu\beta_n^4+(1-\mu)\beta_n^3\gamma_n-(2-\mu){\rm i}\beta_n^2\gamma_n^2
	-(1-\mu)\beta_n\gamma_n^3-\frac{\rm i}{2}\mu\gamma_n^4}{{\rm i}\beta_n^2
	-{\rm i}\gamma_n^2+2\beta_n\gamma_n}\right|^2 |f^{(n)}|^2\\
&\lesssim	\sum\limits_{n\in\mathbb{Z}}
		(1+\alpha_n^2)^{1/2} |f^{(n)}|^2=\|f\|^2_{H^{1/2}(\Gamma_1)},
\end{align*}
and
\begin{align*}
	\|\hat{T}_{22}^{(1)} f\|^2_{H^{-1/2}(\Gamma_1)}&\lesssim\sum\limits_{n\in\mathbb{Z}}
		(1+\alpha_n^2)^{-1/2}\left|\frac{f^{(n)}}{D_1^{(n)}}
		\left(\beta_n^3-{\rm i}\beta_n^2\gamma_n+\beta_n\gamma_n^2
		-{\rm i}\gamma_n^3\right)e^{2\gamma_n\tilde{h}_1}\right| ^2\\
	&\lesssim 	\sum\limits_{n\in\mathbb{Z}}
		(1+\alpha_n^2)^{-1/2}\left|\frac{\left(\beta_n^3-{\rm i}\beta_n^2\gamma_n
		+\beta_n\gamma_n^2-{\rm i}\gamma_n^3\right)e^{2\gamma_n\tilde{h}_1}}{\left({\rm i}\beta_n^2
		-{\rm i}\gamma_n^2+2\beta_n\gamma_n\right)e^{2\gamma_n\tilde{h}_1}}
		\right|^2 |f^{(n)}|^2\\
	&\lesssim	\sum\limits_{n\in\mathbb{Z}}
		(1+\alpha_n^2)^{1/2} |f^{(n)}|^2=\|f\|^2_{H^{1/2}(\Gamma_1)}, 
\end{align*}
which complete the proof.
\end{proof}

The following lemma addresses the error estimates of the DtN operators between the PML problem and the original scattering problem. 

\begin{lemma}\label{Error}
Let
\[
	\Delta^{-}=\min_{n\in\mathbb Z}\left\{\Re(\beta_n)>0\right\},\quad
	\Delta^{+}=\min_{n\in\mathbb Z}\left\{\Im(\beta_n)>0\right\}.
\]
For $k=1, 2$, if the thickness of the PML region $\delta_k$ is sufficiently large, then the following estimates hold for any $f\in H^{3/2}(\Gamma_k)$:
\begin{align*}
\|(\hat{T}_{11}^{(k)}-T_{11}^{(k)})f\|_{H^{-3/2}(\Gamma_k)} &\lesssim  \Theta\|f\|_{H^{3/2}(\Gamma_k)},\\
\|(\hat{T}_{21}^{(k)}-T_{21}^{(k)})f\|_{H^{-1/2}(\Gamma_k)} &\lesssim   \Theta\|f\|_{H^{3/2}(\Gamma_k)},
\end{align*}
and the following estimates hold for any $f\in H^{1/2}(\Gamma_k)$:
\begin{align*}
\|(\hat{T}_{12}^{(k)}-T_{12}^{(k)})f\|_{H^{-3/2}(\Gamma_k)} & \lesssim  \Theta\|f\|_{H^{1/2}(\Gamma_k)},\\
\|(\hat{T}_{22}^{(k)}-T_{22}^{(k)})f\|_{H^{-1/2}(\Gamma_k)} & \lesssim  \Theta\|f\|_{H^{1/2}(\Gamma_k)},
\end{align*}
where 
\begin{equation}\label{Theta}
 \Theta=\max\Big\{e^{- \frac{2\delta_k\sigma_0}{m+1}\Delta^-}, e^{-2\delta_k(\kappa^2+\alpha^2)^{1/2}}, e^{-2\delta_k\Delta^+} \Big\}. 
\end{equation}
\end{lemma}

\begin{proof}
Given the similarity in proof, we only present the details of the error estimate between $\hat{T}_{11}^{(1)}$ and $T_{11}^{(1)}$, with the understanding that the results for the other operators can be obtained in the same manner.

It follows from \eqref{PMLGamma1us} and \eqref{T1} that
\begin{align*}
& (T_{11}^{(1)}-\hat{T}_{11}^{(1)}) f = \sum\limits_{n\in\mathbb{Z}}\frac{e^{{\rm i}\alpha_n x}}{D^{(n)}_1}f_1^{(n)}
\Big\{\left(-2{\rm i}\beta_n^4\gamma_n+4\beta_n^3\gamma_n^2+2{\rm i}\beta_n^2\gamma_n^3\right)\\
&\quad+\left(2{\rm i}\beta_n^4\gamma_n+2\beta_n^3\gamma_n^2+2{\rm i}\beta_n^2\gamma_n^3+2\beta_n\gamma_n^4\right)
e^{2{\rm i}\beta_n \tilde{h}_1} -\left(8{\rm i}\beta_n^2\gamma_n^3+\beta_n^3\gamma_n^2\right)e^{({\rm i}\beta_n+\gamma_n)\tilde{h}_1}\\
&\quad+\left(2\beta_n^3\gamma_n^2+4{\rm i}\beta_n^2\gamma_n^3-2\beta_n\gamma_n^4\right)e^{2({\rm i}\beta_n+\gamma_n) \tilde{h}_1}\Big\}.
\end{align*}
When $\delta_1$ is sufficiently large, the dominant term in the denominator is the one involving $e^{2 \gamma_n \tilde{h}_1}$, which has a coefficient of ${\rm i}\beta_n^2-{\rm i}\gamma_n^2+2\beta_n\gamma_n$. The dominant terms in the numerator are either the constant term or the one that contains $e^{2({\rm i}\beta_n+\gamma_n) \tilde{h}_1}$.
By the choice of PML parameters given in \eqref{h1}, a straightforward calculation yields 
\begin{align*}
2\gamma_h\tilde{h}_1=2\delta_1\left(1+\frac{{\rm i}\sigma_0}{m+1}\right)(\kappa^2+\alpha_n^2)^{1/2}
	=2\delta_1(\kappa^2+\alpha_n^2)^{1/2}+\frac{2{\rm i}\delta_1\sigma_0}{m+1}(\kappa^2+\alpha_n^2)^{1/2},
\end{align*}
and
\begin{align*}
2{\rm i}\beta_n\tilde{h}_1=\left\{
	\begin{aligned}
	&2{\rm i}\delta_1\left(1+\frac{{\rm i}\sigma_0}{m+1}\right)(\kappa^2-\alpha_n^2)^{1/2}\\
	&\hspace{1cm}=-\frac{2\delta_1\sigma_0}{m+1}(\kappa^2-\alpha_n^2)^{1/2}+2{\rm i}\delta_1(\kappa^2-\alpha_n^2)^{1/2}\quad\text{if} ~ |\alpha_n|<\kappa,\\
	&-2\delta_1\left(1+\frac{{\rm i}\sigma_0}{m+1}\right)(\alpha_n^2-\kappa^2)^{1/2}\\
	&\hspace{1cm}=-2\delta_1(\alpha_n^2-\kappa^2)^{1/2}-\frac{2{\rm i}\delta_1\sigma_0}{m+1}(\alpha_n^2-\kappa^2)^{1/2}
\quad\text{if} ~ |\alpha_n|>\kappa.\end{aligned}
	\right.
\end{align*}
Then we obtain 
\begin{align*}
&\|(T_{11}^{(1)}-\hat{T}_{11}^{(1)}) f\|^2_{H^{-3/2}(\Gamma_1)}\lesssim
\sum\limits_{n\in\mathbb{Z}}(1+\alpha_n^2)^{-3/2}\Bigg(\bigg|
	\frac{-2{\rm i}\beta_n^4\gamma_n+4\beta_n^3\gamma_n^2
	+2{\rm i}\beta_n^2\gamma_n^3}{\left({\rm i}\beta_n^2-{\rm i}\gamma_n^2+2\beta_n\gamma_n\right)
	e^{2\gamma_n\tilde{h}_1}}\bigg|^2\\
&\qquad+\bigg| \frac{\left(2\beta_n^3\gamma_n^2+4{\rm i}\beta_n^2\gamma_n^3-2\beta_n\gamma_n^4\right)
	e^{2({\rm i}\beta_n+\gamma_n) \tilde{h}_1}}
	{\left({\rm i}\beta_n^2-{\rm i}\gamma_n^2+2\beta_n\gamma_n\right)e^{2\gamma_n\tilde{h}_1}}\bigg|^2\Bigg)|f_1^{(n)}|^2\\
&\lesssim\sum\limits_{n\in\mathbb{Z}}(1+\alpha_n^2)^{3/2} 
	\left(e^{-4\delta_1(\kappa^2+\alpha_n^2)^{1/2}}+e^{-\frac{4\delta_1\sigma_0}{m+1}|\kappa^2-\alpha_n^2|^{1/2}}
	+e^{-4\delta_1|\alpha_n^2-\kappa^2|^{1/2}}\right)|f_1^{(n)}|^2\\
&\lesssim  \Theta\|f\|_{H^{3/2}(\Gamma_1)},
\end{align*}
which completes the proof. 
\end{proof}

It is evident from Lemma \ref{Error} that the DtN operators of the PML problem exhibit exponential convergence in the operator norm to the DtN operators of the original scattering problem. This convergence is a crucial factor contributing to the exponential convergence of the PML solution towards the solution of the original scattering problem.

\subsection{Convergence analysis}

By employing the TBCs for the PML problem as given in \eqref{DtNGamma1PML} and \eqref{PMLGamma2}, the
PML problem \eqref{MBHE} can be reformulated equivalently  in the domain $\Omega$:
\begin{equation}\label{TotalBiharmonicPML}
\left\{
\begin{aligned}
	&\Delta^2 u^{\rm PML}-\kappa^4 u^{\rm PML}=0 & &{\rm in} ~ \Omega,\\
	& u^{\rm PML}=0,\quad \partial_{\nu} u^{\rm PML}=0 &&{\rm on} ~ \Gamma_c,\\
	& N_1 u^{\rm PML}=\hat{T}_{11}^{(1)} f_1+\hat{T}_{12}^{(1)}g_1+\hat{p}_1 &&{\rm on} ~\Gamma_1,\\
	& M_1 u^{\rm PML}=\hat{T}_{21}^{(1)}f_1+\hat{T}_{22}^{(1)}g_1+\hat{p}_2 &&{\rm on} ~ \Gamma_1,\\
	& N_2 u^{\rm PML}=\hat{T}_{11}^{(2)} f_2+\hat{T}_{12}^{(2)}g_2 &&{\rm on} ~ \Gamma_2,\\
	& M_2 u^{\rm PML}=\hat{T}_{21}^{(2)}f_2+\hat{T}_{22}^{(2)}g_2 &&{\rm on}~ \Gamma_2,
\end{aligned}
\right.
\end{equation}
where $\hat{p}_1$ and $\hat{p}_2$ are given in \eqref{hp12}, $(f_k, g_k)$ represent the Dirichlet data of the total field $u^{\rm PML}$ on $\Gamma_k$ for $k=1, 2$. 

The variational problem of \eqref{TotalBiharmonicPML} is to find $u^{\rm PML}\in H_{{\rm qp}, \Gamma_c}^2(\Omega)$
such that 
\begin{equation}\label{variationalPML}
a^{\rm PML}(u^{\rm PML}, v)=\int_{\Gamma_1} \left(\hat{p}_1\bar{v} 
+ \hat{p}_2\partial_\nu\bar{v}\right){\rm d}s\quad\forall\, v\in H_{{\rm qp}, \Gamma_c}^2(\Omega),
\end{equation}
where the sesquilinear form 
$a^{\rm PML}(u, v):  H_{{\rm qp}, \Gamma_c}^2(\Omega)\times  H_{{\rm qp}, \Gamma_c}^2(\Omega)\rightarrow \mathbb{C}$ is defined as 
\begin{align}\label{auvPML}
a^{\rm PML}(u, v)=\int_{\Omega} \bigg[\mu\Delta u\Delta \bar{v}+(1-\mu)\sum\limits_{i,j=1}^2\frac{\partial^2 u}{\partial x_i\partial x_j}\frac{\partial^2 \bar{v}}{\partial x_i\partial x_j}-\kappa^4 u\bar{v}\bigg]{\rm d}x\notag\\
-\sum_{k=1}^2\int_{\Gamma_k} (\mathbb{\hat T}^{(k)}\boldsymbol u)\cdot\overline{\boldsymbol v}{\rm d}s. 
\end{align}
with $\boldsymbol{u}, \boldsymbol{v}$, and $\mathbb {\hat T}^{(k)}$ given by 
\[
 \boldsymbol{u}=\begin{bmatrix}
                 u\\
                 \partial_\nu u
                \end{bmatrix},\quad \boldsymbol{v}=\begin{bmatrix}
                 v\\
                 \partial_\nu v
                \end{bmatrix},\quad 
\mathbb{\hat T}^{(k)}=\begin{bmatrix}
                 {\hat T}_{11}^{(k)} & {\hat T}_{12}^{(k)}\\[5pt]
                 {\hat T}_{21}^{(k)} & {\hat T}_{22}^{(k)}
                \end{bmatrix}.
\]

\begin{theorem}\label{PositivePML}
Assuming that the thickness of the PML regions is sufficiently large, the variational problem \eqref{variationalPML} has a unique weak solution $u^{\rm PML}\in H^2_{{\rm qp}, \Gamma_c}(\Omega)$ except for a discrete set of wavenumbers $\kappa$. Moreover, the solution satisfies the error estimate
\begin{eqnarray}\label{estimate}
\|u-u^{\rm PML}\|_{H^2(\Omega)}\lesssim  \Theta\|u^i\|_{H^2(\Omega)},
\end{eqnarray}
where $u^i$ is the incident field and $u$ is the solution of the variational problem \eqref{variationalTBC}. 
\end{theorem}

\begin{proof}
First, we demonstrate  that the sesquilinear form \eqref{auvPML} satisfies the G\r{a}rding inequality. It follows from
Lemmas \ref{PositiveTBC}, \ref{traceTheorem}, \ref{BoundPML}--\ref{Error} and the trace theorem that 
\begin{align*}
	\Re a^{\rm PML}(u, u) &= \Re a(u, u)+\Re\left\{\sum_{k=1}^2\int_{\Gamma_k} \big( ({\mathbb T}^{(k)}-\hat{\mathbb T}^{(k)})\boldsymbol u\big)\cdot \overline{\boldsymbol u}{\rm d}s\right\}\\
 	&\geq c_1 \|u\|^2_{H^2(\Omega)}-c_2\|u\|^2_{H^1(\Omega)}
 	- \Theta\left(c_3 \|u\|^2_{H^{3/2}(\Gamma_1\cup\Gamma_2)}+c_4 \|u\|^2_{H^{1/2}(\Gamma_1\cup\Gamma_2)}\right)\\
 	&\geq c_1 \|u\|^2_{H^2(\Omega)}-c_2\|u\|^2_{H^1(\Omega)}
 	- c_5\Theta  \|u\|^2_{H^{2}(\Omega)}. 
\end{align*}
Given the exponential decay of $\Theta$ concerning $\delta_k$, we can choose $\delta_k$ to be sufficiently large to ensure $c_1-c_5\Theta>0$. For all but a possibly discrete set of wavenumbers $\kappa$, it follows from the Fredholm alternative theorem that the variational problem is well-posed. Consequently, there is a positive constant $\gamma$ for which the subsequent inf-sup condition is satisfied:
\[
\sup\limits_{0\neq v\in H^2_{{\rm qp}, \Gamma_c}(\Omega)}
\frac{|a^{\rm PML}(u, v)|}{\|v\|_{H^2(\Omega)}}\geq \gamma \|u\|_{H^2(\Omega)}\quad\forall\, u\in H_{{\rm qp}, \Gamma_c}^2(\Omega). 
\]
Moreover, the PML solution $u^{\rm PML}$ satisfies the stability estimate
\begin{equation}\label{stability}
	\|u^{\rm PML}\|_{H^2(\Omega)}\lesssim \|u^{\rm inc}\|_{H^2(\Omega)}.
\end{equation}

It remains to prove the error estimate \eqref{estimate}. Denote by $e=u^{\rm PML}-u$ the error between the PML solution and the solution to the original scattering problem. Upon a straightforward calculation, we obtain
\begin{align*}
&\int_{\Gamma_1} \left(\left( \hat{p}_1-p_1\right) \bar{v}
	+\left(\hat{p}_2-p_2\right)\partial_\nu\bar{v} \right) {\rm d}s\\
&=a^{\rm PML}(u^{\rm PML}, v)-a(u, v)\\
&=a(e, v)+\sum_{k=1}^2\int_{\Gamma_k} \big(  ({\mathbb T}^{(k)}-\hat{\mathbb T}^{(k)}){\boldsymbol u}^{\rm PML}
	\big)\cdot\overline{\boldsymbol v}{\rm d}s,
\end{align*}
where ${\boldsymbol u}^{\rm PML}=[u^{\rm PML}, \partial_\nu u^{\rm PML}]^\top$. Hence we have 
\begin{align*}
a^{\rm PML}(e, v)&=-\sum_{k=1}^2\int_{\Gamma_k} \big(({\mathbb T}^{(k)}-\hat{\mathbb T}^{(k)}){\boldsymbol u}^{\rm PML}
	\big)\cdot\overline{\boldsymbol v}{\rm d}s-\int_{\Gamma_1} \big(  ({\mathbb T}^{(1)}-\hat{\mathbb T}^{(1)}){\boldsymbol u}^i\big)\cdot\overline{\boldsymbol v}{\rm d}s,
\end{align*}
where ${\boldsymbol u}^i=[u^i, \partial_\nu u^i]^\top$. By utilizing the continuity of the sesquilinear form \eqref{tbcauv}, the stability estimate \eqref{stability},  and Lemma \ref{Error}, we obtain 
\begin{eqnarray*}
\|e\|_{H^2(\Omega)}\lesssim 
\sup\limits_{0\neq v\in H^2_{{\rm qp}, \Gamma_c}(\Omega)}
	\frac{|a^{\rm PML}(e, v)|}{\|v\|_{H^2(\Omega)}}\lesssim \Theta
	 \left(\|u^{\rm PML}\|_{H^2(\Omega)}+\|u^i\|_{H^2(\Omega)}\right)
	 \lesssim \Theta\|u^i\|_{H^2(\Omega)},
\end{eqnarray*}
which completes the proof.
\end{proof}

It is apparent from Theorem \ref{PositivePML} and the definition of $\Theta$ in \eqref{Theta} that as the thickness of the PML regions increases, the PML solution $u^{\rm PML}$ exhibits exponential convergence towards the solution of the original scattering problem $u$.

\section{Conclusion}\label{Section6}

In this paper, we have investigated the scattering of flexural waves by a one-dimensional periodic array of cavities embedded in an infinite elastic thin plate. The problem is formulated using the biharmonic wave equation in an unbounded domain. Initially, TBCs are introduced to reduce the scattering problem into a bounded domain, and the well-posedness of the associated variational problem is examined. Subsequently, the PML method is employed to transform the problem from an unbounded domain to a bounded one. The corresponding TBCs are derived, and the well-posedness of the PML problem is established. Notably, exponential convergence is achieved between the PML solution and the solution to the original scattering problem.

This work is centered on formulating and analyzing the biharmonic wave scattering problem in one-dimensional periodic structures. Currently, we are developing numerical methods, including the finite element method, to solve the PML problem. The progress and results of this ongoing development will be detailed in a forthcoming publication.

\backmatter

%
%
%

\bmhead{Acknowledgments}

The first author is supported partially by National Natural Science Foundation of China (U21A20425) and a Key Laboratory of Zhejiang Province. The second author is supported in part by the NSF grant DMS-2208256. The third author is supported by the NSFC grants 12201245 and 12171017.

\section*{Declarations}
%
%
\begin{itemize}
\item Conflict of interest: On behalf of all authors, the corresponding author states that there is no conflict of 
interest.
\end{itemize}

\bibliography{sn-bibliography}

\end{document}